\documentclass{amsproc}
\usepackage{amsmath}
\usepackage{amssymb}
\usepackage{amsfonts}

\theoremstyle{plain}

\newtheorem{corollary}{Corollary}

\newtheorem{definition}{Definition}

\newtheorem{lemma}{Lemma}
\newtheorem{notation}{Notation}

\newtheorem{proposition}{Proposition}
\newtheorem{remark}{Remark}

\newtheorem{theorem}{Theorem}
\numberwithin{equation}{section}
\input{tcilatex}

\begin{document}
\title[On Spectrum of Nonlinear Continuous Operators]{On Spectrum of
Nonlinear Continuous Operators}
\author{Kamal N. Soltanov}
\address{{\small National Academy of Sciences of Azerbaijan, Baku, AZERBAIJAN%
}}
\email{sultan\_kamal@hotmail.com}
\urladdr{https://www.researchgate.net/profile/Kamal-Soltanov/research}
\date{}
\subjclass[2010]{Primary 47J10, 47H10; Secondary 35P30, 35A01}
\keywords{Nonlinear continuous operator, spectrum, Banach spaces, nonlinear
differential operator, solvability}

\begin{abstract}
This article proposed a new approach to the determination of the spectrum
for nonlinear continuous operators in the Banach spaces and using it
investigated the spectrum of some classes of operators. Here shows that in
nonlinear operators case is necessary to seek the spectrum of a given
nonlinear operator relatively to another nonlinear operator. Moreover, the
order of nonlinearity of examined operator and operator relatively to which
seek the spectrum must be identical. Here provided different examples
relative to how one can find the iegenvalue and also studied solvability
problems.
\end{abstract}

\maketitle

\section{\protect\bigskip Introduction}

Well-known that the spectral theory for linear operators is one of the most
important topics of linear functional analysis, as in many cases for the
study of the linear operator, it is enough to study its spectrum (e.g. it
play an essential role in the theory of linear differential operators). It
should be noted the spectral theory of linear operators has essential
application in the many topics of the natural sciences (moreover, the
spectral theory is at the foundation of quantum mechanics). The essential
information on a linear operator contains in its spectrum, consequently,
knowledge of the spectrum is knowing of many properties of the operator.

As the basic problems of physics, mechanics etc. are nonlinear, at least,
the corresponding differential equations are nonlinear, consequently, these
problems generates nonlinear operators in some Banach spaces. It needs to be
noted in the literature there exist sufficiently many approaches for the
definition of the spectrum for nonlinear continuous operators that beginning
at 60 years of the previous century. In the above-mentioned time were
introduced the various definitions for the spectra for various classes of
nonlinear continuous operators. Unfortunately, many of the introduced
definitions in the really were found only infimum of the spectrum of a
continuous nonlinear operator which allows these authors to study the
nonlinear equations of the form $f\left( \cdot \right) -\lambda _{1}g\left(
\cdot \right) =0$ in the Banach spaces (or in the vector topological spaces)
where $f$ is the basic operator but the operator $g$ is a continuous compact
operator, $\lambda _{1}$ is the same infimum (see, e.g. \cite{3, 5, 6, 9,
10, 11, 24, 36}, etc.). The works \cite{6, 7, 8, 9, 10, 11, 16, 23, 26} were
investigated the Strum- Liuville type problem for the nonlinear perturbed of
linear operators and considered the equation of the form $L=\lambda I+g$ and
bifurcation problem.

Well-known that founding of eigenvalues of nonlinear continuous operators
allow studying of the bifurcations, which appear under the investigation of
the nonlinear equations when it has the form as 
\begin{equation*}
f\left( x\right) -g\left( x\right) =h,\quad x\in X,
\end{equation*}%
where the exponent of the nonlinearity of operator $g$ is greater than the
exponent of the nonlinearity of operator $f$ although $f\succ g$ (see,
Definition \ref{D_2}, Sec. 2), e.g. if both of operators are positive (see,
e.g., \cite{12, 16, 33}\footnote{%
see, also Lopez-Gomez, J. \textit{Spectral Theory and Nonlinear Functional
Analysis} (1st ed.). (2001) Chapman and Hall/CRC, \textit{T\&F Books}
https://doi.org/10.1201/9781420035506\textit{\ }
\par
{}}, etc.).

The works \cite{1, 2, 4, 13, 14, 15, 18, 19, 20, 21, 22, 25, 27, 28} have
been introduced spectra (or first eigenvalue) starting from the equation of
the form $f\left( \cdot \right) -\lambda _{1}I\cdot =0$ as in the theory of
the linear operators. And the spectra for the classes that are Frechet
differentiable operators, Lipschitz continuous operators, continuous
operators, special continuous operators, kept operators, and linearly
bounded operators were investigated. This approach supposed that the
spectrum of the operator $f$ acting in the Banach space $X$ can define such
as in the theory of linear operators. In these works is defined a resolvent
subset of $\mathbb{K}$ and it denoted as $\rho \left( f\right) \subset 
\mathbb{K}$ of the elements $\lambda $ under which the resolvent operator $%
f-\lambda I$ is invertible. Consequently, then a subset $\sigma \left(
f\right) =\mathbb{K-}\rho \left( f\right) $ is called the spectra of the
operator $f:X\longrightarrow X$, where $I\equiv id$ (identical operator).
This approach and obtained results in enough form explained in the book \cite%
{1} (see also the survey \cite{2}). All of the above works for the study was
used degree theory that requires the compactness that later on was
generalized to the condition that uses the Kuratowski measure of
noncompactness. Unfortunately, these definitions couldn't fulfill properties
from the viewpoint of the above requirements since in this case often the
defined $\lambda $ can dependent on elements of the space $X$ (see, e.g.
examples provided in next).

Therefore, is arise the natural question: Could be to introduce a reasonable
definition of a spectrum of a continuous nonlinear operator that satisfies
some basic requirements, which were analogous to the existing properties in
the theory of the linear operators?

But how one will see later on from the explanation of the characters of the
nonlinear operators for the study of the spectrum of the nonlinear operator
one needs to approach by another way.

This paper is proposed a new approach for the study of the spectrum of
continuous nonlinear operators in the Banach spaces. Really here we find the
first eigenvalue of the nonlinear continuous operator in Banach space and
this shows how one can seek the other eigenvalues. Moreover, we investigate
the solvability of the nonlinear equations in the Banach spaces. Here shows
that if use the proposed definition of the spectrum of nonlinear continuous
operators in Banach spaces then the spectra will satisfy some properties,
that are similar to properties having in the theory of the linear operators.

Later on (in Section 2) will be shown that really the obtained numbers
aren't the spectra of the operator as, in general, these can be to depends
on elements of the domain of the examined operator. The founded numbers $%
\lambda _{1}$ allow investigating the solvability of nonlinear equations
that have the form $f\left( \cdot \right) -\lambda g\left( \cdot \right) =y$
containing this continuous nonlinear operator $f$ and another continuous
operator $g$ under some complementary conditions such as $\left\vert \lambda
\right\vert \leq \lambda _{1}$, where $\lambda $ is some number.

In this paper, we will study the spectrum of nonlinear operators acting in
Banach spaces, and also the solvability of the depended on parameters
nonlinear equations using the solvability theorems and fixed-point theorems
of the works \cite{29, 30, 31, 32}. Let $X,Y$ be real Banach spaces on the
field $\Re $ and $X^{\ast },Y^{\ast }$ be of their dual spaces, let $Y$ be
reflexive space. Let $f:X$ $\longrightarrow Y$ and $g:X$ $\longrightarrow Y$
be nonlinear continuous operators such, that $f\left( 0\right) =0$, $g\left(
0\right) =0$.

For investigation of the spectrum of continuous nonlinear operators, will
consider the following equation 
\begin{equation}
f\left( x\right) =\lambda g\left( x\right) ,\quad x\in M\subseteq X,
\label{P}
\end{equation}%
where $f,\ g$ are continuous operators, in particular, $g$ can be the
identical operator, and also to consider the depended on a parameter $%
\lambda $ equation 
\begin{equation}
f_{\lambda }\left( x\right) \equiv f\left( x\right) -\lambda g\left(
x\right) =y,\quad for\ y\in Y,  \label{Pa}
\end{equation}%
where, generally, $\lambda $ is an element of $%
\mathbb{C}
$. Roughly speaking, here is required to study for which $\lambda $ these
equations are solvable.

Our goal is the investigation the spectrum of the continuous nonlinear
operator in a generalized sense, and also study in the Banach spaces the
solvability of the nonlinear operator equations dependent on the parameter.

The paper is organized as follows. Section 2 provided a definition of the
spectrum of nonlinear continuous operators in the Banach spaces, some
complementaries to the definition with explanations and examples, the
general theorems on the solvability, and fixed-point theorem, the
solvability theorem for the dependent on a parameter nonlinear equation with
continuous operators is proved. Here is shown how can be to find the first
eigenvalue of nonlinear continuous operators relative to another nonlinear
continuous operator. Section 3 provided some examples of nonlinear
differential operators, for which founded first eigenvalues relative to
another nonlinear differential operators and showed the relations between
founding first eigenvalues with first eigenvalues of linear differential
operators. Section 4 studied the existence of the first eigenvalues of the
fully nonlinear continuous operator relatively to other nonlinear continuous
operators and provided examples.

\section{Spectral properties of the continuous nonlinear operators}

Here a concept for the spectrum of the continuous nonlinear operator with
respect to another continuous nonlinear operator is introduced, as above
noted we would like to determine such numbers, which are independent of the
element of the space. It isn't difficult to see that if one determines the
spectrum of the continuous nonlinear operator in the same way as for the
linear continuous operator then the finding number will be the function of
elements of the definition domain of the operator.

Let$\ X$ and $Y$ be the real Banach spaces, $F:D\left( f\right)
=X\longrightarrow Y$, $G:X\subseteq D\left( G\right) \longrightarrow Y$ be
the continuous bounded nonlinear operators (for generality) and $\lambda \in 
\mathbb{R}
$ be the number and $F\left( 0\right) =0$, $G\left( 0\right) =0$.

So, we will investigate the spectrum of operator $F$ with respect to
operator $G$, i.e. we will study the question: For which number $\lambda $
the following equation will be solvable? 
\begin{equation}
f_{\lambda }\left( x\right) \equiv F\left( x\right) -\lambda G\left(
x\right) =0,\ \text{or }F\left( x\right) =\lambda G\left( x\right) ,\ x\in X
\label{2.1}
\end{equation}%
And also we will study the following equation 
\begin{equation}
f_{\lambda }\left( x\right) \equiv F\left( x\right) -\lambda G\left(
x\right) =y,\quad y\in Y.  \label{2.1a}
\end{equation}

In the beginning, we will introduce concepts that will be necessary for this
paper.

\begin{definition}
\label{D_1}The operator $f:D\left( f\right) \subseteq X\longrightarrow Y$ is
called bounded if there is a continuous function $\mu
:R_{+}^{1}\longrightarrow R_{+}^{1}$ such that 
\begin{equation*}
\left\Vert f\left( x\right) \right\Vert _{Y}\leq \mu \left( \left\Vert
x\right\Vert _{X}\right) ,\quad \forall x\in D\left( f\right)
\end{equation*}%
and denote this class of operators as $\mathfrak{B}$ and the bounded
continuous class of operators by $\mathfrak{B}C^{0}$.
\end{definition}

Now we introduce an order relation in the class of the continuous operators
acting in the Banach spaces.

\begin{definition}
\label{D_2}Let $X_{0},Y_{0}$ be a Banach spaces and $F:D\left( f\right)
\subseteq X_{0}$ $\longrightarrow Y_{0}$ , $G:D\left( G\right) \subseteq
X_{0}\longrightarrow Y_{0}$ be some continuous operators. Denote by $%
\mathcal{F}_{F}\left( Z\right) $, $\mathcal{F}_{G}\left( Z\right) $ sets
defined in the form 
\begin{equation*}
\mathcal{F}_{F}\left( Z\right) \equiv \left\{ x\in X_{0}\left\vert
~\left\Vert F\left( x\right) \right\Vert _{Z}\right. <\infty \right\} \neq
\varnothing \ \ \&
\end{equation*}%
\begin{equation*}
\mathcal{F}_{G}\left( Z\right) \equiv \left\{ x\in X_{0}\left\vert
~\left\Vert G\left( x\right) \right\Vert _{Z}\right. <\infty \right\} \neq
\varnothing
\end{equation*}%
that are subsets of $X_{0}$ for each Banach space $Z\subseteq Y_{0}$\
satisfying conditions $\func{Im}F\cap Z\neq \varnothing $, $\func{Im}G\cap
Z\neq \varnothing $. If $\mathcal{F}_{F}\left( Z\right) \subset \mathcal{F}%
_{G}\left( Z\right) $ holds for each of the above mentioned Banach space $%
Z\subseteq Y_{0}$, then we will say that the operator $F$ is greater than
the operator $G$ that denote as $F\succ G$.
\end{definition}

\begin{definition}
\label{D_S} Let operators $F:X\longrightarrow Y$ and $G:X\subseteq D\left(
G\right) \longrightarrow Y$, moreover $F\succ G$. We say $\lambda \in 
\mathcal{K}$\ belongs to the $G-$resolvents subset of operator $F$ relative
to operator $G$ iff 
\begin{equation*}
\lambda \in \rho _{G}\left( F\right) \equiv
\end{equation*}%
\begin{equation*}
\left\{ \lambda \in 
\mathbb{R}
\left\vert \ f_{\lambda }^{-1}\equiv \left( F-\lambda G\ \right)
^{-1}:F\left( X\right) \cap G\left( X\right) \subseteq Y\longrightarrow
X,\right. \ f_{\lambda }^{-1}\in \mathfrak{B}C^{0}\right\}
\end{equation*}%
holds and denote this by $\rho _{G}\left( F\right) \subseteq \mathcal{K}$,
where $f_{\lambda }\left( \cdot \right) \equiv F\left( \cdot \right)
-\lambda G\left( \cdot \right) $. Consequently, we call the element $\lambda
\in \mathcal{K}$ is the $G-$spectrum of the operator $F$ if $\lambda \in 
\mathcal{K}-$ $\rho _{G}\left( F\right) $, that we denote by $\sigma
_{G}\left( F\right) $.
\end{definition}

\textbf{Notation. }Unfortunately, the above definition of a spectrum in such
a way does not approach the pair of operators, which are chosen by the
independent way, that will be shown next. We will call the number $\lambda $
is the first eigenvalue of the examined operator relative to another
operator as in Definition \ref{D_S}, which is independent of the elements
from the domain of the examined operators. This definition allows us to seek
also the following eigenvalues of examined operator.

So, for simplicity, we start to consider the case when $F\succ G$ and
consider the case when one of these operators has the inverse operator from
the class $\mathfrak{B}C^{0}$. Let\ operator $F$ be invertible, i.e. there
is $F^{-1}:F\left( X\right) \subseteq Y\longrightarrow X$. Then using $F^{-1}
$ we get the equation 
\begin{equation}
y-\lambda G\left( F^{-1}\left( y\right) \right) =0,\quad y=F\left( x\right)
,\ x\in X,  \label{2.2}
\end{equation}%
that is needed to study on the $F\left( X\right) \subseteq Y$. Thus we
derive an equation that is equivalent to the equation for investigation of
the existence of a spectrum as in the works \cite{1, 2, 3, 4, 6, 9, 13, 18,
19, 20, 22, 23, 25, 28} and their references. Unlike a usual case, here the
operator $G\circ F^{-1}$ is defined on subset $F\left( X\right) $ and acts
as $G\circ F^{-1}:F\left( X\right) \longrightarrow G\left( X\right)
\subseteq Y$. If $G$ is invertible then by the same way as above we get to
equation 
\begin{equation*}
F\left( G^{-1}\left( y\right) \right) -\lambda y=0,\quad y=G\left( x\right)
,\ x\in X,
\end{equation*}%
where $G^{-1}$ denotes the inverse operator of $G$. Consequently, this
equation needs to investigate on subset $G\left( X\right) $ of $Y$.

Thus if $F$ (or $G$) is invertible then we obtain the operator 
\begin{equation}
\widetilde{f}_{\lambda }\left( \cdot \right) \equiv I\cdot -\lambda G\left(
F^{-1}\left( \cdot \right) \right) ,\quad \widetilde{f}_{\lambda }:D\left( 
\widetilde{f}_{\lambda }\right) \subseteq Y\longrightarrow Y  \label{2.3}
\end{equation}%
that depends on the parameter $\lambda $, consequently the equation (\ref%
{2.2}) is transformed to the problem on the study of the eigenvalue of
operator $G\circ F^{-1}$ (or $F\circ G^{-1}$), in other words, we derived
problem about the existence of the fixed-points of operator $G\circ F^{-1}$
(or $F\circ G^{-1})$ in subset $F\left( X\right) $ (or $G\left( X\right) $)
of $Y$.

It is clear that if $F$ is the linear continuous operator with the inverse
operator $F^{-1}$ then the problem (\ref{2.2}) is equivalent to the problem 
\begin{equation}
x-\lambda F^{-1}\circ G\left( x\right) =0  \label{2.4}
\end{equation}%
that becomes the problem about the existence of the fixed-points of the
operator $\lambda F^{-1}\circ G$. In many articles, problems of such types
were studied (\cite{9, 10, 11, 17, 26, 35}, etc.) but we wish to investigate
the problem (\ref{2.1}) in the general case, without such type conditions.
Section 3 will be given some explanations relative to the above-provided
cases.

Before starting the investigation of the spectrum of the nonlinear operators
relative to other nonlinear operators in the general sense necessary to
investigate the solvability of nonlinear equations (\ref{2.1a}). To
investigate these problems, we will use the general existence and
fixed-point theorems of \cite{29, 30}.

Therefore, in the beginning, we will lead these results.

\subsection{General Solvability Results}

Let $X,Y$ be real Banach spaces such as above, $f:D\left( f\right) \subseteq
X\longrightarrow Y$ be an operator and $B_{r_{0}}^{X}\left( 0\right) $ $%
\subseteq D\left( f\right) $ is the closed ball with a center of $0\in X$.

Consider the following conditions.

(\textit{i}) $f:D\left( f\right) \subseteq X\longrightarrow Y$ be a bounded
continuous operator;

(\textit{ii}) There is a mapping $g:X\subseteq D\left( g\right)
\longrightarrow Y^{\ast }$ such that$\ g\left( B_{r_{0}}^{X}\left( 0\right)
\right) =B_{r_{1}}^{Y^{\ast }}\left( 0\right) $ and 
\begin{equation*}
\left\langle f\left( x\right) ,\widetilde{g}\left( x\right) \right\rangle
\geq \nu \left( \left\Vert x\right\Vert _{X}\right) =\nu \left( r\right)
,\quad \forall x\in S_{r}^{X}\left( 0\right)
\end{equation*}%
holds \footnote{%
In particular, the mapping $g$ can be a linear bounded operator as $g\equiv
L:X\longrightarrow Y^{\ast }$ that satisfy the conditions of \textit{(ii). }}%
, where $\widetilde{g}\left( x\right) \equiv \frac{g\left( x\right) }{%
\left\Vert g\left( x\right) \right\Vert }$, $\nu :R_{+}^{1}\longrightarrow
R^{1}$ and $\nu \left( r_{0}\right) \geq \delta _{0}$ is a continuous
function ($\nu \in C^{0}$), moreover $\nu \left( \tau \right) $ is the
nondecreasing function for $\tau :$ $\tau _{0}\leq \tau \leq r_{0}$; $\delta
_{0}>0$, $\tau _{0}\geq 0$ are a constant.

(\textit{iii}) Almost each $x_{0}\in int\ B_{r_{0}}^{X}\left( 0\right) $
possess such neighborhood $V_{\varepsilon }\left( x_{0}\right) $, $%
\varepsilon \geq \varepsilon _{0}$ that the following inequality 
\begin{equation}
\left\Vert f\left( x_{2}\right) -f\left( x_{1}\right) \right\Vert _{Y}\geq
\Phi \left( \left\Vert x_{2}-x_{1}\right\Vert _{X},x_{0},\varepsilon \right)
\label{2.6}
\end{equation}%
holds for any $\forall x_{1},x_{2}\in V_{\varepsilon }\left( x_{0}\right)
\cap B_{r_{0}}^{X}\left( 0\right) $, where $\varepsilon _{0}>0$\ and $\Phi
\left( \tau ,x_{0},\varepsilon \right) \geq 0$ is the continuous function of 
$\tau $ and $\Phi \left( \tau ,\widetilde{x},\varepsilon \right)
=0\Leftrightarrow \tau =0$ (in particular, maybe $x_{0}=0$, $\varepsilon
=\varepsilon _{0}=r_{0}$ and $V_{\varepsilon }\left( x_{0}\right)
=V_{r_{0}}\left( 0\right) \equiv B_{r_{0}}^{X}\left( 0\right) $, consequently%
$\ \Phi \left( \tau ,x_{0},\varepsilon \right) \equiv \Phi \left( \tau
,x_{0},r_{0}\right) $ on $B_{r_{0}}^{X}\left( 0\right) $).

\begin{theorem}
\label{Th_1}Let $X,Y$ be real Banach spaces such as above, $f:D\left(
f\right) \subseteq X\longrightarrow Y$ be an operator and $%
B_{r_{0}}^{X}\left( 0\right) $ $\subseteq D\left( f\right) $ is the closed
ball with a center of $0\in D\left( f\right) $. Assume conditions (i) and
(ii) are fulfilled. Then the image $f\left( B_{r_{0}}^{X}\left( 0\right)
\right) $ of the ball $B_{r_{0}}^{X}\left( 0\right) $ contains in an
absorbing subset of $Y$\ an everywhere dense subset of $M$ that is defined
as follows 
\begin{equation*}
M\equiv \left\{ y\in Y\left\vert \ \left\langle y,\widetilde{g}\left(
x\right) \right\rangle \leq \left\langle f\left( x\right) ,\widetilde{g}%
\left( x\right) \right\rangle ,\right. \forall x\in S_{r_{0}}^{X}\left(
0\right) \right\} .
\end{equation*}

Furthermore, if the condition (iii)\ also is fulfilled then the image $%
f\left( B_{r_{0}}^{X}\left( 0\right) \right) $ of the ball $%
B_{r_{0}}^{X}\left( 0\right) $ is a bodily subset of $Y$, moreover $%
B_{\delta _{0}}^{Y}\left( 0\right) \subseteq M$.
\end{theorem}

The proof of this theorem, and also its generalization was provided in \cite%
{29} (see also, \cite{30, 31, 32}]).\footnote{%
We note Theorem 2 is the generalization of Theorem of such type from
\par
Soltanov K.N., On equations with continuous mappings in Banach spaces.
Funct. Anal. Appl. (1999) 33, 1, 76-81.}

The condition (\textit{iii}) can be generalized, for example, as in the
following proposition.

\begin{corollary}
\label{C_1}Let all conditions of Theorem \ref{Th_1} be fulfilled except for
the inequality (\ref{2.6}) of condition (iii) instead the following
inequality 
\begin{equation}
\left\Vert f\left( x_{2}\right) -f\left( x_{1}\right) \right\Vert _{Y}\geq
\Phi \left( \left\Vert x_{2}-x_{1}\right\Vert _{X},x_{0},\varepsilon \right)
+\psi \left( \left\Vert x_{1}-x_{2}\right\Vert _{Z},x_{0},\varepsilon \right)
\label{2.7}
\end{equation}%
holds, where $Z$ is Banach space and $X\subset Z$ is compact, $\psi \left(
\cdot ,x_{0},\varepsilon \right) :R_{+}^{1}\longrightarrow R^{1}$ is a
continuous function relatively $\tau \in R_{+}^{1}$ and $\psi \left(
0,x_{0},\varepsilon \right) =0$.

Then the statement of Theorem \ref{Th_1} is true.
\end{corollary}

From Theorem 1 immediately follows

\begin{theorem}
\label{Th_2}(\textbf{Fixed-Point Theorem}). Let $X$ be a real reflexive
separable Banach space and $f_{1}:D\left( f_{1}\right) \subseteq
X\longrightarrow X$ be a bounded continuous operator. Moreover, let on
closed ball $B_{r_{0}}^{X}\left( 0\right) \subseteq D\left( f_{1}\right) $,
with the center of $0\in D\left( f_{1}\right) $, operators $f_{1}$ and $%
f\equiv Id-f_{1}$ satisfy the following conditions:

(I) The following inequations 
\begin{equation*}
\left\Vert f_{1}\left( x\right) \right\Vert _{X}\leq \mu \left( \left\Vert
x\right\Vert _{X}\right) ,\quad \ \forall x\in B_{r_{0}}^{X}\left( 0\right) ,
\end{equation*}%
\begin{equation}
\left\langle f\left( x\right) ,\widetilde{g}\left( x\right) \right\rangle
\geq \nu \left( \left\Vert x\right\Vert _{X}\right) ,\quad \forall x\in
B_{r_{0}}^{X}\left( 0\right) ,  \label{2.8}
\end{equation}%
hold, where $f_{1}\left( B_{r_{0}}^{X}\left( 0\right) \right) \subseteq
B_{r_{0}}^{X}\left( 0\right) $, $g:D\left( g\right) \subseteq
X\longrightarrow X^{\ast }$, $D\left( f_{1}\right) \subseteq D\left(
g\right) $ and satisfy the condition (ii) (in particular, $g\equiv
J:X\rightleftarrows X^{\ast }$, i.e. $g$ be a duality mapping), $\mu $ and $%
\nu $ are such functions as in Theorem \ref{Th_1};

(II) Almost each $x_{0}\in intB_{r_{0}}^{X}\left( 0\right) $ possess a
neighborhood $V_{\varepsilon }\left( x_{0}\right) $, $\varepsilon \geq
\varepsilon _{0}>0$ such that for each $x_{0}\in int\ B_{r_{0}}^{X}\left(
0\right) $ the following inequality%
\begin{equation*}
\left\Vert f\left( x_{2}\right) -f\left( x_{1}\right) \right\Vert _{X}\geq
\varphi \left( \left\Vert x_{2}-x_{1}\right\Vert _{X},x_{0},\varepsilon
\right) ,
\end{equation*}%
holds for any $\forall x_{1},x_{2}\in V_{\varepsilon }\left( x_{0}\right)
\cap B_{r_{0}}^{X}\left( 0\right) $, where the function $\varphi \left( \tau
,x_{0},\varepsilon \right) $ satisfies the condition such as conditions on
functions of right part of (\ref{2.7}).

Then operator $f_{1}$ possess a fixed-point in the closed ball $%
B_{r_{0}}^{X}\left( 0\right) $.
\end{theorem}

Now we introduce the following concept.

\begin{definition}
\label{D_3}An operator $f:D\left( f\right) \subseteq X\longrightarrow Y$
possesses the \textrm{P-property} iff each precompact subset $M\subseteq 
\func{Im}f$ of $Y$ contains such subsequence (maybe generalized) $%
M_{0}\subseteq M$ that $f^{-1}\left( M_{0}\right) \subseteq G$ and $%
M_{0}\subseteq f\left( G\cap D\left( f\right) \right) $, where $G$ is a
precompact subset of $X$.
\end{definition}

\begin{notation}
\label{N_1}It is easy to see that the condition (iii) of Theorem \ref{Th_1}
one can replace by the condition: $f$ possesses \textrm{P-property}.

It should be noted if $f^{-1}$ is a lower or upper semi-continuous mapping
then operator $f:D\left( f\right) \subseteq X\longrightarrow Y$ possesses of
the \textrm{P-property}.
\end{notation}

In the above results, condition \textit{(iii)} is required for the
completeness of the image of the considered operator $f$. One can describe
and other complementary conditions on $f$ under which $\func{Im}f$ will be a
closed subset (see, \cite{30, 31,32}). In particular, the following results
are true.

\begin{lemma}
\label{L_1}Let $X,Y$ be Banach spaces such as above, $f:D\left( f\right)
\subseteq X\longrightarrow Y$ be a bounded continuous operator, and $D\left(
f\right) $ is a weakly closed subset of the reflexive space $X$. Let $f$
have a weakly closed graph and for each bounded subset $M\subset Y$ the
subset $f$ $^{-1}\left( M\right) $ is the bounded subset of $X$. Then $f$ is
a weakly closed operator.
\end{lemma}

We want to note the graph of operator $f$ is weakly closed iff from $x_{m}%
\overset{X}{\rightharpoonup }x_{0}\in D\left( f\right) $ and $f\left(
x_{m}\right) \overset{Y}{\rightharpoonup }y_{0}\in Y$ follows equation $%
f\left( x_{0}\right) \equiv y_{0}\func{Im}f$ $\subset Y$ (for the general
case see \cite{30,31}).

For the proof is enough to note, if $\left\{ y_{m}\right\} _{m=1}^{\infty
}\subset \func{Im}f\subset Y$ is the weakly convergent sequence of $Y$ then $%
f^{-1}\left( \left\{ y_{m}\right\} _{m=1}^{\infty }\right) $ is a bounded
subset of $X$ consequently this has such subsequence $\left\{ x_{m}\right\}
_{m=1}^{\infty }$ that $x_{m}\in f^{-1}\left( y_{m}\right) $ and $x_{m}%
\overset{X}{\rightharpoonup }x_{0}\in D\left( f\right) $ for some element $%
x_{0}\in D\left( f\right) $ by virtue of the reflexivity of $X$.

\begin{lemma}
\label{L_2}Let $X,Y$ be reflexive Banach spaces, and $f:D\left( f\right)
\subseteq X\longrightarrow Y$ be a bounded continuous mapping that satisfies
the condition: if $G\subseteq D\left( f\right) $ is a closed convex subset
of $X$ then $f\left( G\right) $ is the weakly closed subset of $Y$. Then if $%
G\subseteq D\left( f\right) $ is a bounded closed convex subset of $X$ then $%
f\left( G\right) $ is a closed subset of $Y$.
\end{lemma}

For the proof enough to use the reflexivity of the space $X$ and properties
of the bounded closed convex subset of $X$ (see, e. g. \cite{29, 30}).

\begin{lemma}
\label{L_3}Let $X$ be a Banach space such as above, $f:X\longrightarrow
X^{\ast }$ be a monotone operator satisfying conditions of Theorem \ref{Th_1}%
, and $r\geq \tau _{1}$ be some number. Then $f\left( G\right) $ is a
bounded closed subset containing a ball $B_{r_{1}}^{X^{\ast }}\left( f\left(
0\right) \right) $ for every such bounded closed convex body $G\subset X$
that $B_{r}^{X}\left( 0\right) \subset G$, where $r_{1}=r_{1}\left( r\right)
\geq \delta _{1}>0$.
\end{lemma}

\subsection{Investigation of equations (\protect\ref{2.1}), (\protect\ref%
{2.1a}) and existence of the spectra}

We start with the study of the equation (\ref{2.1a}), in order to explain
the role of the number $\lambda $ under the investigation of posed
questions. Let $X,Y$ be real reflexive Banach spaces, $F:X$ $\longrightarrow
Y$ , $G:X\subseteq D\left( G\right) $ $\longrightarrow Y$ are nonlinear
operators and $B_{r_{0}}^{X}\left( 0\right) $ ($r_{0}>0$) be a closed ball
with a center at $0\in X$ that belongs to $D\left( F\right) $. Since in this
work, we will consider only operators acting in real spaces, therefore will
seek real numbers $\lambda _{0}$, under which the considered equation may be
solvable.

Assume on the ball $B_{r_{0}}^{X}\left( 0\right) $ are fulfilled the
following conditions:

1) $F:B_{r_{0}}^{X}\left( 0\right) \longrightarrow Y$ , $G:B_{r_{0}}^{X}%
\left( 0\right) $ $\longrightarrow Y$ are a bounded continuous operators,
i.e. there exist such continuous functions $\mu _{j}:\Re _{+}\longrightarrow 
$ $\Re _{+}$, $j=1,2$ that 
\begin{equation*}
\left\Vert F\left( x\right) \right\Vert _{Y}\leq \mu _{1}\left( \left\Vert
x\right\Vert _{X}\right) ;\quad \left\Vert G\left( x\right) \right\Vert
_{Y}\leq \mu _{2}\left( \left\Vert x\right\Vert _{X}\right) ,
\end{equation*}%
hold for any $x\in B_{r_{0}}^{X}\left( 0\right) $, in addition $F\succ G$ ;

2) Let $f_{\lambda }\equiv F-\lambda G$ be the operator of (\ref{2.1a}).
Assume there exists such parameter $\lambda _{0}\in 
\mathbb{R}
_{+}$ that for each $\left( y^{\ast },r,\lambda \right) $\ exists such $x\in
S_{r}^{X}\left( 0\right) $ that the following inequality 
\begin{equation*}
\left\langle f_{\lambda }\left( x\right) ,y^{\ast }\right\rangle \geq \nu
_{\lambda }\left( \left\Vert x\right\Vert _{X}\right) ,\quad \exists x\in
S_{r}^{X}\left( 0\right) ,\quad g\left( x\right) =y^{\ast }
\end{equation*}%
holds,\ where $\left( y^{\ast },r,\left\vert \lambda \right\vert \right) \in
S_{1}^{Y^{\ast }}\left( 0\right) \times \left( 0,r_{0}\right] \times $ $%
\left( 0,\lambda _{0}\right] $ and $\nu _{\lambda }:\Re _{+}\longrightarrow
\Re $ is the continuous function satisfying the condition $\left( ii\right) $
of Theorem \ref{Th_1}, in this case $\delta _{0}=\delta _{0\lambda }\searrow
0$ if $\left\vert \lambda \right\vert \nearrow \left\vert \lambda
_{0}\right\vert $.

3) Assume for almost every point $x_{0}$ from $B_{r_{0}}^{X}\left( 0\right) $
there exist numbers $\varepsilon \geq \varepsilon _{0}>0$ and such
continuous of $\tau $ functions $\varphi _{\lambda }\left( \tau
,x_{0},\varepsilon \right) \geq 0$, $\psi _{\lambda }(\tau
,x_{0},\varepsilon )$ that the following inequality 
\begin{equation*}
\left\Vert f_{\lambda }\left( x_{1}\right) -f_{\lambda }\left( x_{2}\right)
\right\Vert _{Y}\geq \varphi _{\lambda }\left( \left\Vert
x_{1}-x_{2}\right\Vert _{X},x_{0},\varepsilon \right) +\psi _{\lambda
}(\left\Vert x_{1}-x_{2}\right\Vert _{Z},x_{0},\varepsilon )
\end{equation*}%
holds for any $x_{1},x_{2}\in B_{\varepsilon }^{X}\left( x_{0}\right) $,
where $\varphi \left( \tau ,x_{0},\varepsilon \right) =0\Leftrightarrow \tau
=0$, $\psi _{\lambda }(\cdot ,x_{0},\varepsilon ):\Re _{+}\longrightarrow $ $%
\Re $, $\psi _{\lambda }(0,x_{0},\varepsilon )=0$ for any $\left(
x_{0},\varepsilon \right) $ and $Z$ be a such Banach space that the
inclusion $X\subset Z$ is compact.

\begin{theorem}
\label{Th_3}Let conditions 1, 2 and 3 are fulfilled on the closed ball $%
B_{r_{0}}^{X}\left( 0\right) \subset X$. Then equation (\ref{2.1a}) is
solvable for $\forall \widetilde{y}\in V_{\lambda }\subset Y$ and each $%
\lambda :0\leq \left\vert \lambda \right\vert \leq \lambda _{0}$; moreover,
the inclusion $B_{\delta _{0}}^{Y}\left( 0\right) \subseteq f_{\lambda
}\left( B_{r_{0}}^{X}\left( 0\right) \right) $ holds for $\delta _{0}\equiv
\delta _{0}\left( \lambda \right) >0$ from the condition 2, where $%
V_{\lambda }$ can be defined as follows 
\begin{equation*}
V_{\lambda }\equiv \left\{ \left. \widetilde{y}\in Y\right\vert \
\left\langle \widetilde{y},g\left( x\right) \right\rangle \leq \left\langle
f_{\lambda }\left( x\right) ,g\left( x\right) \right\rangle ,\quad \forall
x\in S_{r_{0}}^{X}\left( 0\right) \right\} .
\end{equation*}
\end{theorem}

For the proof is sufficient to note that all conditions of Theorem \ref{Th_1}
are fulfilled for each fixed $\lambda :$ $\left\vert \lambda \right\vert
<\lambda _{0}$ due to conditions of Theorem \ref{Th_3}, therefore applying
Theorem \ref{Th_1} we get the correctness of Theorem \ref{Th_3}.

Consequently, the equation (\ref{2.3}) also is solvable in $B_{r}^{X}\left(
0\right) $ under the conditions on $\widetilde{f}_{\lambda }$ of the above
type that depends at $\lambda _{0}$, e.g. 
\begin{equation*}
\left\Vert G\left( F^{-1}\left( y_{1}\right) \right) -G\left( F^{-1}\left(
y_{2}\right) \right) \right\Vert _{Y}\leq C\left( x_{0},\varepsilon \right)
\left\Vert y_{1}-y_{2}\right\Vert _{Y}+\psi _{\lambda }(\left\Vert
y_{1}-y_{2}\right\Vert _{Z},y_{0},\varepsilon ),
\end{equation*}%
where $C\left( x_{0},\varepsilon \right) \lambda _{0}<1$ and the inclusion $%
Y\subset Z$ is compact.

Whence follows using Theorem \ref{Th_2} one can obtain the solvability of
the equation (\ref{2.1a}). Really, let $Y=X^{\ast }$ and closed ball $%
B_{r_{0}}^{X}\left( x_{0}\right) $ ($r_{0}>0$) belongs to $D\left( F\right) $%
. Let condition 1 of Theorem \ref{Th_3} is fulfilled on ball $%
B_{r_{0}}^{X}\left( x_{0}\right) $. Assume the following conditions are
fulfilled.

2%
\'{}%
) There exists such parameter $\lambda _{1}\in 
\mathbb{R}
$ that $\lambda _{1}G\left( F^{-1}\left( F\left( B_{r_{0}}^{X}\left(
x_{0}\right) \right) \right) \right) \subseteq B_{r_{0}}^{X}\left(
x_{0}\right) $ and for each $x^{\ast }\in S_{1}^{X^{\ast }}\left( 0\right) $%
\ there exists such $x\in S_{r}^{X}\left( x_{0}\right) $, for each $r\in
\left( 0<r\leq r_{0}\right] $ that the following inequality 
\begin{equation*}
\left\langle \widetilde{f}_{\lambda _{1}}\left( x\right) ,x^{\ast
}\right\rangle \geq \nu _{\lambda _{1}}\left( \left\Vert x-x_{0}\right\Vert
_{X}\right) =\nu _{\lambda _{1}}\left( r\right) ,\quad x\in S_{r}^{X}\left(
x_{0}\right) \subset B_{r_{0}}^{X}\left( x_{0}\right)
\end{equation*}%
holds, where $\nu _{\lambda _{1}}:\Re _{+}\longrightarrow \Re $ is the
continuous function that satisfies the condition $\left( ii\right) $ of
Theorem \ref{Th_1}

3%
\'{}%
) For almost every $\widehat{x}\in B_{r_{0}}^{X}\left( x_{0}\right) $ there
are such number $\varepsilon \geq \varepsilon _{0}>0$ and continuous
functions $\Phi _{\lambda _{1}}(\cdot ,\widehat{x},\varepsilon ):\Re
_{+}\longrightarrow $ $\Re _{+}$, $\varphi _{\lambda }(\cdot ,\widehat{x}%
,\varepsilon ):\Re _{+}\longrightarrow $ $\Re $ for each $\left( \widehat{x}%
,\varepsilon \right) $, that the following inequality 
\begin{equation*}
\left\Vert \widetilde{f}_{\lambda _{1}}\left( x_{1}\right) -\widetilde{f}%
_{\lambda _{1}}\left( x_{2}\right) \right\Vert _{X}\geq \Phi _{\lambda
_{1}}\left( \left\Vert x_{1}-x_{2}\right\Vert _{X},\widehat{x},\varepsilon
\right) +\varphi _{\lambda _{1}}\left( \left\Vert x_{1}-x_{2}\right\Vert
_{Z},\widehat{x},\varepsilon \right)
\end{equation*}%
holds for any $x_{1},x_{2}\in U_{\varepsilon }\left( \widehat{x}\right) \cap
B_{r_{0}}^{X}\left( x_{0}\right) $, where $\Phi _{\lambda _{1}}\left( \tau ,%
\widehat{x},\varepsilon \right) \geq 0$ and $\Phi _{\lambda _{1}}\left( \tau
,\widehat{x},\varepsilon \right) =0\Leftrightarrow \tau =0$, $\varphi
_{\lambda _{1}}\left( 0,\widehat{x},\varepsilon \right) =0$, and also $Z$ be
a Banach space, in addition $X\subset Z$ is compact.

Whence implies, that for defined above $\lambda _{1}$ all conditions of
Theorem \ref{Th_1} are fulfilled for the operator $\widetilde{f}_{\lambda
_{1}}$ on the closed ball $B_{r_{0}}^{X}\left( x_{0}\right) $. Consequently, 
$\widetilde{f}_{\lambda _{1}}\left( B_{r_{0}}^{X}\left( x_{0}\right) \right) 
$ contains a closed absorbing subset of $X$ (at least, $0\in X$) by virtue
of the Theorem \ref{Th_1}. In the other words, $0\in \widetilde{f}_{\lambda
_{1}}\left( B_{r_{0}}^{X}\left( x_{0}\right) \right) $ and therefore there
exists an element $\widetilde{x}\in B_{r_{0}}^{X}\left( x_{0}\right) $ for
which $f_{\lambda _{1}}\left( \widetilde{x}\right) =0$ holds, i. e. $F\left( 
\widetilde{x}\right) =\lambda _{1}G\left( \widetilde{x}\right) $.

The obtained result one can formulate as follows.

\begin{corollary}
\label{C_2}Let $F,G$ be above determined operators, $F\succ G$, $D\left(
F\right) \subseteq D\left( G\right) $, and there exists such number $\lambda
_{1}$ that conditions 1), 2%
\'{}%
), 3%
\'{}%
) are fulfilled on the closed ball $B_{r}^{X}\left( x_{0}\right) \subseteq
D\left( F\right) \subseteq X$. Then there exists an element $\widetilde{x}%
\in B_{r}^{X}\left( x_{0}\right) $ such, that $F\left( \widetilde{x}\right)
=\lambda _{1}G\left( \widetilde{x}\right) $ or $\lambda _{1}G\left(
F^{-1}\left( \cdot \right) \right) $ has fixed point.
\end{corollary}

Let $X,Y$ be Banach spaces and\ $B_{r_{0}}^{X}\left( 0\right) \subseteq
D\left( F\right) \subset X$, $r_{0}>0$, $F\succ G$, $F\left( 0\right) =0$, $%
G\left( 0\right) =0$ be a bounded operators and there are the continuous
functions $\nu _{F},\nu _{G}:\Re _{+}\longrightarrow \Re $ satisfying the
condition $\left( ii\right) $ of Theorem \ref{Th_1} such that for each $%
y^{\ast }\in S_{1}^{Y^{\ast }}\left( 0\right) $\ there exists $x\in
S_{r}^{X}\left( 0\right) $ for which the inequalities \ 
\begin{equation*}
\left\langle F\left( x\right) ,y^{\ast }\right\rangle \geq \nu _{F}\left(
\left\Vert x\right\Vert _{X}\right) ,\quad \left\langle G\left( x\right)
,y^{\ast }\right\rangle \geq \nu _{G}\left( \left\Vert x\right\Vert
_{X_{1}}\right)
\end{equation*}%
hold, where $X_{1}$ is the Banach space that $X\subseteq X_{1}$ (we denote
the relation between $x$ and $y^{\ast }\in S_{1}^{Y^{\ast }}\left( 0\right) $
by $g:S_{r}^{X}\left( 0\right) \longrightarrow S_{1}^{Y^{\ast }}\left(
0\right) $, $0<r\leq $\ $r_{0}$, so that $g\left( x\right) =y^{\ast }$).
Then according to condition 2%
\'{}
of Corollary \ref{C_2}, one may expect that spectrum of the operator $%
F:D\left( F\right) \subset X\longrightarrow Y$ relative to operator $%
G:D\left( G\right) \subseteq X\longrightarrow Y$ can define in the following
way 
\begin{equation}
\lambda =\inf \left\{ \frac{\left\langle F\left( x\right) ,g\left( x\right)
\right\rangle }{\left\langle G\left( x\right) ,g\left( x\right)
\right\rangle }\left\vert \ x\in \right. B_{r_{0}}^{X}\left( 0\right)
\backslash \left\{ 0\right\} \right\} ,\quad r_{0}>0.  \label{2.5}
\end{equation}

Can we call the determined by (\ref{2.5}) $\lambda $ spectrum of the
operator $G\circ F^{-1}$ or spectrum of operator $F$ relative to operator $G$%
? Generally speaking, one cannot name since the composition $G\circ F^{-1}$
can be nonlinear and $\lambda _{1}$ may be a function as $\lambda
_{1}=\lambda _{1}\left( x_{1}\right) $, unlike the linear case, where $x_{1}$
is the element on which the relation (\ref{2.5}) attains infimum. Moreover,
if we define the subspace $\Gamma _{\lambda _{1}}=\left\{ \alpha
x_{1}\left\vert \ \alpha \in R\right. \right\} \subset X$ then for $\alpha
x_{1}\in D\left( F\right) $, generally, $\alpha \lambda _{1}x_{1}\neq G\circ
F^{-1}\left( \alpha x_{1}\right) $ since $G\circ F^{-1}$ is nonlinear
operator.

Indeed, if the power of nonlinearity of the operator $F$ is great than the
power of nonlinearity of operator $G$, or the inverse of its, then
obviously, will the case $\lambda _{1}=\lambda _{1}\left( x_{1}\right) $.
For example, operators $F$ and $G$ defines as the following form 
\begin{equation*}
F\left( u\right) =-\nabla \circ \left( \left\vert \nabla u\right\vert
^{p_{0}-2}\nabla u\right) ,\quad G\left( u\right) =\left\vert u\right\vert
^{p_{1}-2}u,\quad Y=W^{-1,q}\left( \Omega \right) ,
\end{equation*}%
where $X=W_{0}^{1,p_{0}}\left( \Omega \right) \cap L^{p_{1}}\left( \Omega
\right) $, $\Omega \subset R^{n}$, $n\geq 1$, with sufficiently smooth
boundary $\partial \Omega $ and $p=\max \left\{ p_{0},p_{1}\right\} $, $%
p_{0},p_{1}>2$, $q=p`=\frac{p}{p-1}$. \ Assume $p_{0}\neq p_{1}$ and $%
F:D\left( F\right) =W^{1,p_{0}}\left( \Omega \right) \longrightarrow $ $%
W^{-1,q_{0}}\left( \Omega \right) $, $G:D\left( G\right) =L^{p_{1}}\left(
\Omega \right) \longrightarrow $ $L^{q_{1}}\left( \Omega \right) $. Then
using (\ref{2.5}) we get 
\begin{equation*}
\lambda =\inf \left\{ \frac{\left\langle F\left( u\right) ,u\right\rangle }{%
\left\langle G\left( u\right) ,u\right\rangle }\left\vert \ u\in \right.
B_{r_{0}}^{X}\left( 0\right) \backslash \left\{ 0\right\} \right\} =
\end{equation*}%
\begin{equation*}
\inf \left\{ \frac{\left\Vert \nabla u\right\Vert _{L^{p_{0}}}^{p_{0}}}{%
\left\Vert u\right\Vert _{L^{p_{1}}}^{p_{1}}}\left\vert \ u\in
B_{r_{0}}^{W_{0}^{1,p_{0}}\cap L^{p_{1}}\left( \Omega \right) }\left(
0\right) \backslash \left\{ 0\right\} \right. \right\} .
\end{equation*}%
Whence we have if $p_{0}>p_{1}$ then 
\begin{equation*}
\lambda =\inf \left\{ \left( \frac{\left\Vert \nabla u\right\Vert
_{L^{p_{0}}}}{\left\Vert u\right\Vert _{L^{p_{1}}}}\right)
^{p_{1}}\left\Vert \nabla u\right\Vert _{L^{p_{0}}}^{p_{0}-p_{1}}\left\vert
\ u\in B_{r_{0}}^{W_{0}^{1,p_{0}}\cap L^{p_{1}}\left( \Omega \right) }\left(
0\right) \backslash \left\{ 0\right\} \right. \right\}
\end{equation*}%
and if $p_{0}<p_{1}$ then 
\begin{equation*}
\lambda =\inf \left\{ \left( \frac{\left\Vert \nabla u\right\Vert
_{L^{p_{0}}}}{\left\Vert u\right\Vert _{L^{p_{1}}}}\right)
^{p_{0}}\left\Vert u\right\Vert _{L^{p_{1}}}^{p_{0}-p_{1}}\left\vert \ u\in
B_{r_{0}}^{W_{0}^{1,p_{0}}\cap L^{p_{1}}\left( \Omega \right) }\left(
0\right) \backslash \left\{ 0\right\} \right. \right\} .
\end{equation*}

Consequently, $\lambda $ will be a function of $u$ as $\lambda _{1}=\lambda
\left( u_{1}\right) $, here $u_{1}$ is the function on which is attained the
infimum of the above-mentioned expression. (Section 3 has more examples.)

\begin{remark}
\label{R_1}Whence implies the main parts of operators $F$ and $G$ must
possess a common degree of nonlinearity in order to $\lambda _{0}$ couldn't
be a function of $x$. From Theorem \ref{Th_1} follows that defined in such
way number $\lambda _{0}$ is the number that was assumed exists in the
conditions of this theorem. Moreover, the finding number allows us in
mentioned theorem to state the existence of solutions for each $\lambda
:0\leq \left\vert \lambda \right\vert \leq \lambda _{0}$ if the element on
the right side is from a determined subsets. Thus, we can find the spectrum
in the sense Definition \ref{D_S} using the proof of the above result, if
the main parts of operators $F$ and $G$ have a common degree of nonlinearity.
\end{remark}

The spectrum of an operator usually must be to characterize the operator,
but the determined here number $\lambda $ can't of this. Therefore we will
fit the above question differently unlike the above-mentioned works.

In the beginning, we will study the case when the defined in (\ref{2.5})
number $\lambda $ is independent of $x$. As was explained above from the
existence of the fixed point $x_{1}$ of the operator $G\circ F^{-1}$ does
not follow that $x_{1}$ is the eigenvector and the number $\lambda _{1}$ is
the eigenvalue for this operator. Consequently, the existence of the inverse
operator of $F$ or $G$ also does not resolve this question. All of these
shows that for the study of the posed question is necessary some relations
between the operators $F$ and $G$.

So, we assume one of the following conditions are fulfilled: (1) $F$ and $G$
are homogeneous with common exponent $p>0$ or a common function $\phi \left(
\cdot \right) $, i.e. $F\left( \mu x\right) \equiv \mu ^{p}F\left( x\right) $%
, $G\left( \mu x\right) \equiv \mu ^{p}G\left( x\right) $ or $F\left( \mu
x\right) \equiv \phi \left( \mu \right) F\left( x\right) $, $G\left( \mu
x\right) \equiv \phi \left( \mu \right) G\left( x\right) $ for any $\mu >0$;
(2) Investigate the problem locally, i.e. study the problem on the closed
ball $B_{r}^{X}\left( 0\right) \subseteq D\left( f_{\lambda }\right) $ for
selected $r>0$ and to seek of $\lambda $ in the form $\lambda \equiv \lambda
\left( r\right) $. \ 

We start to study case (1), i.e. when $F$ and $G$ are homogeneous with
exponent $p>0$. Whence imply that (\ref{2.5}) defines a number $\lambda $
independent from $x$ hence if denote this minimum by $\lambda _{1}$ and the
element at which minimum is attained by $x_{1}$ then (\ref{2.5}) will
fulfill for $\forall x\in $ $\Gamma _{\lambda _{1}}\cap D\left( F\right) $.
Consequently, in this case, one can define $x_{1}$ as the first eigenvector
and $\lambda _{1}$ as the first eigenvalue of operator $F$ relatively to
operator $G$ (as in the linear case), or one can define as the fixed point $%
y_{1}=F\left( x_{1}\right) $ of operator $\lambda G\circ F^{-1}$.

Now let be the case (2). Then if $F$ and $G$ have of the different orders of
the homogeneous that are given by different functions, e.g. by polynomial
functions with exponents $p_{F}\neq p_{G}$ then possible 2 variants: \textit{%
(a) } $p_{F}>p_{G}$ and \textit{(b) }$p_{F}<p_{G}$.

If the case (a) occur then $F\left( x\right) =r^{p_{F}}F\left( \widetilde{x}%
\right) $ and $G\left( x\right) =r^{p_{G}}G\left( \widetilde{x}\right) $
since for any $x\in X_{0}$ one can write $x\equiv r\widetilde{x}\ $ where $%
\left\Vert x\right\Vert _{X_{0}}\equiv r$ and $\widetilde{x}=\frac{x}{r}\in
S_{1}^{X_{0}}\left( 0\right) \subset X_{0}$. Hence due to Theorem \ref{Th_3}
we get that $G$ only can be the perturbation of the operator $F$, therefore
this case not is essential. Let be the case (b). In this case, if there
exists such $\lambda _{0}$ and $x_{0}$ that $F\left( x_{0}\right) =\lambda
_{0}G\left( x_{0}\right) $ then 
\begin{equation*}
F\left( x_{0}\right) =r_{0}^{p_{F}}F\left( \widetilde{x}_{0}\right) ,\
G\left( x\right) =r_{0}^{p_{G}}G\left( \widetilde{x}_{0}\right)
\Longrightarrow r_{0}^{p_{F}}F\left( \widetilde{x}_{0}\right) =\lambda
_{0}r_{0}^{p_{F}}G\left( \widetilde{x}_{0}\right)
\end{equation*}%
\begin{equation*}
\Longrightarrow F\left( \widetilde{x}_{0}\right) =\lambda \left( \lambda
_{0},r_{0}\right) G\left( \widetilde{x}_{0}\right) \Longrightarrow \lambda
\left( \lambda _{0},r_{0}\right) =\lambda _{0}r_{0}^{p_{_{G}}-p_{_{F}}}
\end{equation*}%
holds. Hence follows, that if we change $x_{0}\equiv r_{0}\ \widetilde{x}%
_{0} $ to $x_{1}\equiv r_{1}\ \widetilde{x}_{0}$ then $\lambda $ will change
to $\lambda =\lambda _{0}r_{1}^{p_{_{G}}-p_{_{F}}}$. In other words from
here implies if $p_{F}\neq p_{G}$ then any existing number $\lambda $ will
depend on element $x\in X$, i.e. $\lambda =\lambda \left( r\right) $ on the
line $\left\{ x\in X\left\vert \ x=r\ \widetilde{x}_{0},\right. r\in 
\mathbb{R}
\right\} $. The previous discussion shows there are two variants either $%
p_{F}=p_{G}$ or $\lambda _{0}=\lambda _{0}\left( x_{0}\right) $ and will be
sufficient to investigate these cases. If to assume the operator $f$ is
linear and $g$ is an identical operator then firstly one needs to
investigate the number $\lambda $ belonging to the spectrum of the linear
operator $f$. Well-known in the case when $\lambda $ less than the first
eigenvalue (or less than of the $\inf \left\{ \left\vert \lambda \right\vert
\left\vert \ \lambda \in \sigma \left( f\right) \right. \right\} $) then one
can use the resolvent of the operator f for the study of the nonhomogeneous
equation.

So, here we will study the posed question mainly in the case when condition $%
p_{F}=p_{G}$ holds.

Consequently, the concept defined in the articles \cite{1, 2, 3, 11, 13, 18,
19, 22, 25, 27, 28} etc of the semilinear spectral set is special case of
the Definition \ref{D_S}, by virtue of (\ref{2.2}) and (\ref{2.3}).

\section{Some Application of General Results}

Consider the following problems 
\begin{equation}
-\nabla \circ \left( \left\vert \nabla u\right\vert ^{p-2}\nabla u\right)
-\lambda \left\vert u\right\vert ^{p_{0}-2}u\left\vert \nabla u\right\vert
^{p_{1}}=0,\quad u\left\vert _{\ \partial \Omega }\right. =0,\ \lambda \in 
\mathbb{C}
,  \label{3.1}
\end{equation}

\begin{equation}
-\nabla \circ \left( \left\vert u\right\vert ^{p-2}\nabla u\right) -\lambda
\left\vert u\right\vert ^{p_{0}-2}u=0,\quad u\left\vert _{\ \partial \Omega
}\right. =0,\ \lambda \in 
\mathbb{C}
,  \label{3.2}
\end{equation}%
where $\Omega \subset 
\mathbb{R}
^{n}$ is an open bounded domain with sufficiently smooth boundary $\partial
\Omega $, $n\geq 1$, $p_{0}+p_{1}=p$ and $\nabla \equiv \left(
D_{1},...,D_{n}\right) $. We denote this operator by $f_{0}$, which acts
from $W_{0}^{1.p}\left( \Omega \right) $ to $W^{-1,q}\left( \Omega \right) $%
. It is easy to see that $f_{0}:W_{0}^{1.p}\left( \Omega \right)
\longrightarrow W^{-1,q}\left( \Omega \right) $ is a continuous operator and
for any $u\in W_{0}^{1.p}\left( \Omega \right) $ 
\begin{equation*}
0=\left\langle f_{0}\left( u\right) ,u\right\rangle \equiv \left\langle
-\nabla \circ \left( \left\vert \nabla u\right\vert ^{p-2}\nabla u\right)
-\lambda \left\vert u\right\vert ^{p_{0}-2}u\left\vert \nabla u\right\vert
^{p_{1}},u\right\rangle =
\end{equation*}%
\begin{equation*}
\left\Vert \nabla u\right\Vert _{p}^{p}-\underset{\Omega }{\int }\lambda
\left\vert u\right\vert ^{p_{0}}\left\vert \nabla u\right\vert ^{p_{1}}dx\
\Longrightarrow \left\Vert \nabla u\right\Vert _{p}^{p}=\lambda \underset{%
\Omega }{\int }\left\vert u\right\vert ^{p_{0}}\left\vert \nabla
u\right\vert ^{p_{1}}dx\ 
\end{equation*}%
holds. Whence follows $\lambda \geq 0$ since both of these expressions are
positive.

(1) We will investigate of problem (\ref{3.1}) by using of the Theorem \ref%
{Th_1} or Theorem \ref{Th_3} but we interest to study the question on the
spectrum therefore here we will use Corollary \ref{C_2}. According to the
previous section, we can introduce the following denotations 
\begin{equation*}
F\left( u\right) =-\nabla \left( \left\vert \nabla u\right\vert ^{p-2}\nabla
u\right) ,\quad F:W_{0}^{1.p}\left( \Omega \right) \longrightarrow
W^{-1,q}\left( \Omega \right) ,
\end{equation*}%
\begin{equation*}
G\left( u\right) =\left\vert u\right\vert ^{p_{0}-2}u\left\vert \nabla
u\right\vert ^{p_{1}},\quad G:W_{0}^{1.p}\left( \Omega \right)
\longrightarrow W^{-1,q}\left( \Omega \right) .
\end{equation*}

So, it needs to seek the minimal value of $\lambda $ and function $%
u_{\lambda }\left( x\right) $ (if it exists) for which the equality 
\begin{equation*}
\left\Vert \nabla u\right\Vert _{p}^{p}=\lambda \underset{\Omega }{\int }%
\left\vert u\right\vert ^{p_{0}}\left\vert \nabla u\right\vert
^{p_{1}}dx,\quad u\in B_{1}^{W_{0}^{1.p}\left( \Omega \right) }\left(
0\right)
\end{equation*}%
or the equality 
\begin{equation*}
\lambda =\frac{\left\Vert \nabla u\right\Vert _{p}^{p}}{\underset{\Omega }{%
\int }\left\vert u\right\vert ^{p_{0}}\left\vert \nabla u\right\vert
^{p_{1}}dx}=\ \underset{\Omega }{\int }\left( \frac{\left\Vert \nabla
u\right\Vert _{p}}{\left\vert u\right\vert }\right) ^{p_{0}}\left( \frac{%
\left\Vert \nabla u\right\Vert _{p}}{\left\vert \nabla u\right\vert }\right)
^{p_{1}}dx
\end{equation*}%
holds. Consequently, we need to find the following number 
\begin{equation}
\lambda _{1}=\inf \left\{ \frac{\left\Vert \nabla u\right\Vert _{p}^{p}}{%
\underset{\Omega }{\int }\left\vert u\right\vert ^{p_{0}}\left\vert \nabla
u\right\vert ^{p_{1}}dx}\left\vert \ u\in B_{1}^{\overset{0}{W}\
^{1.p}}\left( 0\right) \right. \right\} .  \label{3.3a}
\end{equation}

It is clear that $\lambda _{1}$ exists and $\lambda _{1}>0$.

Whence follows 
\begin{equation}
\lambda _{1}\geq \left( \frac{\left\Vert \nabla u\right\Vert _{p}}{%
\left\Vert u\right\Vert _{p}}\right) ^{p_{0}}\Longrightarrow \lambda _{1}^{%
\frac{1}{p_{0}}}\geq \frac{\left\Vert \nabla u\right\Vert _{p}}{\left\Vert
u\right\Vert _{p}}  \label{3.3}
\end{equation}%
for any $u\in W_{0}^{1.p}\left( \Omega \right) $, $u\left( x\right) \neq 0$.

We denote by $\lambda _{p_{0},p_{1}}$ the first spectrum of posed problem
that one can define as 
\begin{equation}
\lambda _{p_{0},p_{1}}=\inf \left\{ \left\Vert \nabla u\right\Vert _{p}\left[
\underset{\Omega }{\int }\left\vert u\right\vert ^{p_{0}}\left\vert \nabla
u\right\vert ^{p_{1}}dx\right] ^{-\frac{1}{p}}\left\vert \ u\in
S_{1}^{W_{0}^{1.p}\left( \Omega \right) }\left( 0\right) \right. \right\} .
\label{3.4}
\end{equation}%
From (3.3) we obtain 
\begin{equation*}
\lambda _{p_{0},p_{1}}\geq \lambda _{1}^{\frac{1}{p_{0}}}=\inf \left\{ \frac{%
\left\Vert \nabla u\right\Vert _{p}}{\left\Vert u\right\Vert _{p}}\
\left\vert \ u\in W_{0}^{1.p}\left( \Omega \right) \right. \right\}
\end{equation*}%
that is well-known $\lambda _{1}^{\frac{1}{p_{0}}}=\lambda _{1}\left(
L_{p}\right) $ was defined as the first spectrum of $p-$Laplacian (see, e.g. 
\cite{23}) that is

\begin{equation*}
\lambda _{1}\left( -\Delta _{p}\right) =\inf \ \left\{ \frac{\left\Vert
\nabla u\right\Vert _{p}}{\left\Vert u\right\Vert _{p}}\ \left\vert \ u\in
W_{0}^{1.p}\left( \Omega \right) \right. \right\} ,
\end{equation*}%
consequently, this inequality shows that $\lambda _{p_{0},p_{1}}$ is
comparable with the spectrum $\lambda _{1}\left( -\Delta _{p}\right) $ of
the $p-$Laplacian, i.e. $\lambda _{p_{0},p_{1}}$ satisfy the inequality $%
\lambda _{p_{0},p_{1}}\geq $ $\lambda _{1}\left( -\Delta _{p}\right) $.

(2) Now we will consider of the problem (\ref{3.2}) then it is get 
\begin{equation*}
\underset{\Omega }{\int }\left\vert u\right\vert ^{p-2}\left\vert \nabla
u\right\vert ^{2}dx=\lambda \underset{\Omega }{\int }\left\vert u\right\vert
^{p_{0}}dx
\end{equation*}%
or 
\begin{equation*}
\frac{4}{p^{2}}\left\Vert \nabla \left( \left\vert u\right\vert ^{\frac{p-2}{%
2}}u\right) \right\Vert _{2}^{2}=\lambda \left\Vert \left\vert u\right\vert
^{\frac{p_{0}-2}{2}}u\right\Vert _{2}^{2},
\end{equation*}%
here if assume $p_{0}=p$ and $\left\vert u\right\vert ^{\frac{p-2}{2}%
}u\equiv v$ then we get 
\begin{equation*}
\frac{4}{p^{2}}\left\Vert \nabla v\right\Vert _{2}^{2}=\frac{4}{p^{2}}%
\left\Vert \nabla \left( \left\vert u\right\vert ^{\frac{p-2}{2}}u\right)
\right\Vert _{2}^{2}=\lambda \left\Vert \left\vert u\right\vert ^{\frac{p-2}{%
2}}u\right\Vert _{2}^{2}=\lambda \left\Vert v\right\Vert _{2}^{2}.
\end{equation*}

Whence follows, that the first eigenvalue $\lambda _{1}\left( p\right) $ of
the operator $-\nabla \cdot \left( \left\vert u\right\vert ^{p-2}\nabla
u\right) $ relative to the operator $\left\vert u\right\vert ^{p-2}u$ can be
defined using the first eigenvalue $\lambda _{1}\left( -\Delta \right) $ of
Laplacian that is defined by expression 
\begin{equation*}
\lambda _{1}\left( -\Delta \right) =\inf \ \left\{ \frac{\left\Vert \nabla
v\right\Vert _{2}}{\left\Vert v\right\Vert _{2}}\ \left\vert \ v\in
W_{0}^{1.2}\left( \Omega \right) \right. \right\} .
\end{equation*}%
Consequently, the first eigenvalue $\lambda _{1}\left( p\right) $ of the
operator $-\nabla \cdot \left( \left\vert u\right\vert ^{p-2}\nabla u\right) 
$ relative to the operator $\left\vert u\right\vert ^{p-2}u$ (in appropriate
to problem (\ref{3.2})) one can define by the equality $\lambda _{1}\left(
p\right) =\left( \frac{2}{p}\lambda _{1}\left( -\Delta \right) \right) ^{2}$.

Thus we get

\begin{proposition}
\label{Pr_1} (1) Let $\Delta _{p}$ is the $p-$Laplacian operator with
homogeneous boundary conditions on the bounded domain of $R^{n}$ with smooth
boundary $\partial \Omega $ and $p_{0}+p_{1}=p$. Then the first eigenvalue $%
\lambda _{p_{0}p_{1}}$ for the operator $-\Delta _{p}$ relative to the
operator $G:G\left( u\right) \equiv \left\vert u\right\vert
^{p_{0}-2}u\left\vert \nabla u\right\vert ^{p_{1}}$ exists and be defined
with equality (\ref{3.4}).

(2) If the $F$ and $G$ are operators generated by the problem (\ref{3.2})
and $p_{0}=p$ then the first eigenvalue of the operator $F$ relative to the
operator $G$ is determined as $\lambda _{1}\left( p\right) =\left( \frac{2}{p%
}\lambda _{1}\left( L\right) \right) ^{2}$, where $\lambda _{1}\left(
-\Delta \right) $ is the first eigenvalue of Laplacian.
\end{proposition}

\begin{remark}
\label{R_2}It should be noted by the same way one can define the spectrum of
operator $\Delta _{p}$ (and also of the operator $F:F\left( u\right) \equiv 
\underset{i=1}{\overset{n}{\sum }}D_{i}\left( \left\vert D_{i}u\right\vert
^{p-2}D_{i}u\right) $) relative to operator $G_{0}:G_{0}\left( u\right)
\equiv \underset{i=1}{\overset{n}{\sum }}\left\vert u\right\vert
^{p_{0}-2}u\left\vert D_{i}u\right\vert ^{p_{1}}$.
\end{remark}

Now, using the previous results we will investigate the solvability of the
problem with the following equation in the case when the homogeneous
boundary condition 
\begin{equation}
f_{\lambda }\left( u\right) \equiv -\nabla \left( \left\vert \nabla
u\right\vert ^{p-2}\nabla u\right) -\lambda \left\vert u\right\vert
^{p_{0}-2}u\left\vert \nabla u\right\vert ^{p_{1}}=h\left( x\right)
\label{3.5}
\end{equation}%
on the open bounded domain $\Omega \subset R^{n}$ with smooth boundary $%
\partial \Omega $. For this problem the following result holds.

\begin{theorem}
\label{Th_4}Let numbers $p$, $p_{0}$, $p_{1}\geq 0$ are such that $%
p_{0}+p_{1}=p\geq 2$, $\lambda _{p_{0},p_{1}}$ is the number defined in (\ref%
{3.4}). Then if $\lambda <\lambda _{p_{0},p_{1}}$ then the posed problem is
solvable in $W_{0}^{1.p}\left( \Omega \right) $ for each $h\in
W^{-1,q}\left( \Omega \right) $.
\end{theorem}

\begin{proof}
Let 
\begin{equation*}
f_{\lambda }\left( u\right) \equiv -\nabla \left( \left\vert \nabla
u\right\vert ^{p-2}\nabla u\right) -\lambda \left\vert u\right\vert
^{p_{0}-2}u\left\vert \nabla u\right\vert ^{p_{1}}
\end{equation*}%
is the operator generated by the posed problem for (\ref{3.5}) acts as $%
f_{\lambda }:X\longrightarrow Y$, where $X\equiv W_{0}^{1.p}\left( \Omega
\right) $, $Y\equiv W^{-1,q}\left( \Omega \right) $ according to the above
conditions. For the proof the solvability of this problem we will use of
Corollary \ref{C_1}.

Whence follows that the inequality \ 
\begin{equation*}
\left\langle f_{\lambda }\left( u\right) ,u\right\rangle \equiv \left\Vert
\nabla u\right\Vert _{p}^{p}-\lambda \underset{\Omega }{\int }\left\vert
u\right\vert ^{p_{0}}\left\vert \nabla u\right\vert ^{p_{1}}dx\geq
\end{equation*}%
\begin{equation*}
\left\Vert \nabla u\right\Vert _{p}^{p}-\lambda \left\Vert u\right\Vert
_{p}^{p_{0}}\left\Vert \nabla u\right\Vert _{p}^{p_{1}}=\left\Vert \nabla
u\right\Vert _{p}^{p_{1}}\left( \left\Vert \nabla u\right\Vert
_{p}^{p_{0}}-\lambda \left\Vert u\right\Vert _{p}^{p_{0}}\right) =
\end{equation*}%
\begin{equation*}
\left\Vert \nabla u\right\Vert _{p}^{p}\left( 1-\frac{\lambda }{\lambda
_{p_{0},p_{1}}}\right) =\lambda _{p_{0},p_{1}}^{-1}\left( \lambda
_{p_{0},p_{1}}-\lambda \right) \left\Vert \nabla u\right\Vert _{p}^{p}
\end{equation*}%
holds for any $u\in W_{0}^{1.p}\left( \Omega \right) $ under conditions of
Theorem \ref{Th_4}.

Consequently, if $\lambda <\lambda _{p_{0},p_{1}}$ then $f_{\lambda }$
satisfies the condition \textit{(ii)} of the Theorem \ref{Th_1}, moreover it
is fulfilled for $x_{0}=0$ and $g\equiv Id$. The realization of the
condition (i) of Theorem \ref{Th_1} for $f_{\lambda }$ is obvious.

We provide some inequalities, which show the fulfillment of the condition 
\textit{(iii) }of the Theorem \ref{Th_1} (Corollary \ref{C_1}) for this
problem. It isn't difficult to see 
\begin{equation*}
\left\langle f\left( u\right) -f\left( v\right) ,u-v\right\rangle \equiv
\left\langle \left( \left\vert \nabla u\right\vert ^{p-2}\nabla u-\left\vert
\nabla v\right\vert ^{p-2}\nabla v\right) ,\nabla \left( u-v\right)
\right\rangle -
\end{equation*}%
\begin{equation*}
\lambda \left\langle \left( \left\vert u\right\vert ^{p_{0}-2}\left\vert
\nabla u\right\vert ^{p_{1}}u-\left\vert v\right\vert ^{p_{0}-2}\left\vert
\nabla v\right\vert ^{p_{1}}v\right) ,u-v\right\rangle \geq c_{0}\left\Vert
\nabla \left( u-v\right) \right\Vert _{p}^{p}-
\end{equation*}%
\begin{equation*}
\lambda \left\langle \left\vert u\right\vert ^{p_{0}-2}u\left( \left\vert
\nabla u\right\vert ^{p_{1}}-\left\vert \nabla v\right\vert ^{p_{1}}\right)
,u-v\right\rangle -\lambda \left\langle \left\vert \nabla v\right\vert
^{p_{1}}\left( \left\vert u\right\vert ^{p_{0}-2}u-\left\vert v\right\vert
^{p_{0}-2}v\right) ,u-v\right\rangle
\end{equation*}%
hold for any $u,v\in W_{0}^{1.p}\left( \Omega \right) $. Here the second
term of the right side one can estimate as 
\begin{equation*}
\left\vert \left\langle \left\vert u\right\vert ^{p_{0}-2}u\left( \left\vert
\nabla u\right\vert ^{p_{1}}-\left\vert \nabla v\right\vert ^{p_{1}}\right)
,u-v\right\rangle \right\vert \leq c_{1}\left\langle \left\vert u\right\vert
^{p_{0}-1}\left\vert \nabla \widetilde{u}\right\vert ^{p_{1}-1}\left\vert
\nabla u-\nabla v\right\vert ,\left\vert u-v\right\vert \right\rangle \leq
\end{equation*}%
\begin{equation*}
\varepsilon \left\Vert \nabla \left( u-v\right) \right\Vert _{p}^{p}+C\left(
\varepsilon \right) \left\Vert u\right\Vert _{p}^{\left( p_{0}-1\right)
p^{\prime }}\left\Vert \nabla \widetilde{u}\right\Vert _{p}^{\left(
p_{1}-1\right) p^{\prime }}\left\Vert u-v\right\Vert _{p}^{p^{\prime }},\
p^{\prime }=\frac{p}{p-1}.
\end{equation*}%
Thus we get that all of the conditions of Theorem \ref{Th_1} (the case of
Corollary \ref{C_1}) are fulfilled for the problem (\ref{3.5}).
Consequently, applying Theorem \ref{Th_1} we get the correctness of Theorem %
\ref{Th_4}.
\end{proof}

\section{Fully Nonlinear Operator}

Now we will study the question on the existence of the spectrum of the fully
nonlinear operator. Let $X,Y,Z$ are the real Banach spaces, the inclusion $%
Y\subset Z^{\ast }$is continuous and dense, where $Z^{\ast }$ is the dual
space of $Z$. Let $L:D\left( L\right) \subseteq X\longrightarrow Y$ is
linear operator, $D\left( L\right) $ is dense in $X$; $f:D\left( f\right)
\subseteq Y\longrightarrow Z$ and $g:X\subseteq D\left( f\right)
\longrightarrow Z$ are nonlinear operators. Assume that $L\left( D\left(
L\right) \right) \subseteq $ $D\left( f\right) $.

We wish define the spectrum of the operator $f\circ L:D\left( L\right)
\subseteq X\longrightarrow Z$ with respect to the operator $g$, where $%
D\left( L\right) \subseteq D\left( g\right) $. (We should to noted the
nonlinearity of operator $g$ depends on the natures of nonlinearity of the
operator $f$.)

We will consider problem

\begin{equation}
f\left( Lx\right) -\lambda g\left( x\right) =h,\quad h\in Z,\   \label{4.1a}
\end{equation}%
where $\lambda \in \mathbb{%
\mathbb{R}
}$ be a parameter and $h$ be an element of $Z$.

We will investigate two questions:

(a) Do can define the spectrum of operator $f\circ L$ with respect to
operator $g$?

(b) For which $\lambda \in 
\mathbb{R}
$ and $h\in Z$ equation (\ref{4.1a}) is solvable?

We should be noted our goal is to study the spectrum of the nonlinear
operator in the sense of Definition \ref{D_S}, but as above noted the subset
determined according to this definition contains numbers that depend on the
elements of the domain of the examined operator. Therefore, we will call a
number is the spectrum of the nonlinear operator if it characterizes the
examined operator as the theory of the linear operators. Whence follows
according to the above explanation that necessary to assume the identical
homogeneity of the nonlinearities of the operators $f$ and $g$.

So, we will study the following particular case that in some sense maybe to
explain the general case. To investigate the posed problems we will use the
general results of the articles \cite{29, 30}. Let $B_{r_{0}}^{X}\left(
0\right) \subset D\left( L\right) $, $r_{0}>0$ and consider the following
conditions:

1) There are such constants $c_{1},c_{2}>0$ that $c_{1}\left\Vert
x\right\Vert _{X}\geq \left\Vert Lx\right\Vert _{Y}\geq c_{2}\left\Vert
x\right\Vert _{X}$ for any $x\in D\left( L\right) \subseteq X$, moreover, $Y$
is reflexive space and the inverse to $L$ is a compact operator;

2) The operator $f\circ L$ is greater than the operator $g$, i.e. $f\circ
L\succ g$ ; $f$ and $g$ as the functions are continuous and satisfy the
following conditions: $f\left( t\right) \cdot t>0$ for $\forall t\in
R\setminus \left\{ 0\right\} $; $f\left( 0\right) =0$, $g\left( 0\right) =0$;

3) There is such number $\lambda _{0}>0$ that for each $z^{\ast }\in
S_{1}^{Z^{\ast }}\left( 0\right) $ there exist $x\left( z^{\ast }\right) \in
S_{r}^{X}\left( 0\right) $, $0\leq r\leq r_{0}$ such that the following
inequality 
\begin{equation*}
\left\langle f\left( Lx\right) -\lambda g\left( x\right) ,z^{\ast
}\right\rangle \geq \nu \left( \left\Vert Lx\right\Vert _{Y},\lambda \right)
\end{equation*}%
holds for $\forall \lambda :\left\vert \lambda \right\vert <\lambda _{0}$,
where $\nu :R_{+}\longmapsto R$ is a continuous function and there exists $%
\delta _{0}\left( \lambda \right) >0$ such that $\nu \left( t,\lambda
\right) \geq \delta _{0}\left( \lambda \right) $ for $\forall t=\left\Vert
Lx\right\Vert _{Y}$ when the variable $x$ run over the sphere $%
S_{r_{0}}^{X}\left( 0\right) $;

4) There is such $\varepsilon _{0}>0$ that for $\varepsilon \geq \varepsilon
_{0}>0$ and a. e. $x\in B_{r_{0}}^{X}\left( 0\right) \subseteq X$ there
exists a neighborhood $U_{\varepsilon }\left( x\right) $ such that for $%
\forall x_{1},x_{2}\in U_{\varepsilon }\left( x\right) $ 
\begin{equation*}
\left\langle f\left( Lx_{1}\right) -f\left( Lx_{2}\right)
,Lx_{1}-Lx_{2}\right\rangle \geq l\left( x_{1},x_{2}\right) \left\Vert
Lx_{1}-Lx_{2}\right\Vert _{Y}^{2}
\end{equation*}%
and 
\begin{equation*}
\left\Vert g\left( x_{1}\right) -g\left( x_{2}\right) \right\Vert _{Z}\leq
l_{1}\left( x_{1},x_{2}\right) \left\Vert x_{1}-x_{2}\right\Vert _{X}
\end{equation*}%
hold, where $l\left( x_{1},x_{2}\right) >0$, $l_{1}\left( x_{1},x_{2}\right)
>0$ are a bounded functionals in the sense of similar to Definition \ref{D_1}%
;

\begin{theorem}
\label{Th_5}Let conditions 1-4 be fulfilled and $D\left( L\right) =X$. Then
for each $h\in M\subset Z$ and $\forall \lambda :\left\vert \lambda
\right\vert \leq \lambda _{0}$ equation (\ref{4.1a}) solvable, where a
subset $M$ determined by the following inequation 
\begin{equation*}
M\left( \lambda \right) \subseteq \left\{ z\in Z\left\vert \ \left\vert
\left\langle z,z^{\ast }\right\rangle \right\vert \leq \nu \left( \left\Vert
Lx\right\Vert _{Y},\lambda \right) ,\ \right. \forall z^{\ast }\in
S_{1}^{Z^{\ast }}\left( 0\right) \right\} ,
\end{equation*}%
for each $x\in S_{r_{0}}^{X}\left( 0\right) $.
\end{theorem}

\begin{proof}
For the proof, it is sufficient to show, that the examined operator
satisfies all conditions of the general result of Subsection 2.1. It is
clear the operator $F_{\lambda }\left( x\right) \equiv f\left( Lx\right)
-\lambda g\left( x\right) $ satisfies conditions (\textit{i}),(\textit{ii})
of the general results with the $x_{0}=0$ according to conditions 1-4 (since 
$F_{\lambda }\left( 0\right) =0$). It remains to show fulfill of the
condition \textit{(iii)} and for this sufficiently to investigate the
following expression 
\begin{equation*}
\left\Vert F_{\lambda }\left( x_{1}\right) -F_{\lambda }\left( x_{2}\right)
\right\Vert _{Z}=\left\Vert \left( f\left( Lx_{1}\right) -\lambda g\left(
x_{1}\right) \right) -\left( f\left( Lx_{2}\right) -\lambda g\left(
x_{2}\right) \right) \right\Vert _{Z}.
\end{equation*}%
Now we will prove this expression satisfies the following inequality 
\begin{equation*}
\left\Vert F_{\lambda }\left( x_{1}\right) -F_{\lambda }\left( x_{2}\right)
\right\Vert _{Z}\geq c\left( l\left( x_{1},x_{2}\right) \left\Vert
x_{1}-x_{2}\right\Vert _{X},\lambda \right) -c_{1}\left( l_{1}\left(
x_{1},x_{2}\right) \left\Vert x_{1}-x_{2}\right\Vert _{X_{0}},\lambda
\right) .
\end{equation*}%
We set the following expression: 
\begin{equation*}
\left\langle F_{\lambda }\left( x_{1}\right) -F_{\lambda }\left(
x_{2}\right) ,Lx_{1}-Lx_{2}\right\rangle =\left\langle \left( f\left(
Lx_{1}\right) -f\left( Lx_{2}\right) \right) ,Lx_{1}-Lx_{2}\right\rangle -
\end{equation*}%
\begin{equation*}
-\lambda \left\langle g\left( x_{1}\right) -g\left( x_{2}\right)
,Lx_{1}-Lx_{2}\right\rangle
\end{equation*}%
that is defined correctly, since $F:D\left( L\right) \subseteq
X\longrightarrow Z$ and $L:D\left( L\right) \subseteq X\longrightarrow
Y\subset Z^{\ast }$, whence by carry out certain necessary operations and
considering the conditions of this section we get 
\begin{equation*}
\left\langle F_{\lambda }\left( x_{1}\right) -F_{\lambda }\left(
x_{2}\right) ,Lx_{1}-Lx_{2}\right\rangle =\left\langle f\left( Lx_{1}\right)
-f\left( Lx_{2}\right) ,Lx_{1}-Lx_{2}\right\rangle -
\end{equation*}%
\begin{equation*}
-\left\langle \lambda g\left( x_{1}\right) -\lambda g\left( x_{2}\right)
,Lx_{1}-Lx_{2}\right\rangle \geq l\left( x_{1},x_{2}\right) \left\Vert
Lx_{1}-Lx_{2}\right\Vert _{Y}^{2}-
\end{equation*}%
\begin{equation*}
-\left\vert \lambda \right\vert \left\Vert g\left( x_{1}\right) -g\left(
x_{2}\right) \right\Vert _{Z}\left\Vert L\left( x_{1}-x_{2}\right)
\right\Vert _{Y}\geq l\left( x_{1},x_{2}\right) \left\Vert L\left(
x_{1}-x_{2}\right) \right\Vert _{Y}^{2}-
\end{equation*}%
\begin{equation}
-\left\vert \lambda \right\vert \left\Vert g\left( x_{1}\right) -g\left(
x_{2}\right) \right\Vert _{Z}\left\Vert L\left( x_{1}-x_{2}\right)
\right\Vert _{Y}  \label{4.2a}
\end{equation}%
according to condition 1). Now again of taking into account condition 1) we
arrive to 
\begin{equation}
\left\vert \left\langle F_{\lambda }\left( x_{1}\right) -F_{\lambda }\left(
x_{2}\right) ,Lx_{1}-Lx_{2}\right\rangle \right\vert \leq \left\Vert
F_{\lambda }\left( x_{1}\right) -F_{\lambda }\left( x_{2}\right) \right\Vert
_{Z}\cdot \left\Vert Lx_{1}-Lx_{2}\right\Vert _{Y}.  \label{4.3a}
\end{equation}%
Thus using inequalities (\ref{4.2a}) and (\ref{4.3a}) and condition 4 we get
the following estimate 
\begin{equation}
\left\Vert F_{\lambda }\left( x_{1}\right) -F_{\lambda }\left( x_{2}\right)
\right\Vert _{Z}\geq l\left( x_{1},x_{2}\right) \left\Vert L\left(
x_{1}-x_{2}\right) \right\Vert _{Y}-\left\vert \lambda \right\vert
\left\Vert g\left( x_{1}\right) -g\left( x_{2}\right) \right\Vert _{Z}\geq
\label{4.4a}
\end{equation}%
\begin{equation*}
l\left( x_{1},x_{2}\right) \left\Vert L\left( x_{1}-x_{2}\right) \right\Vert
_{Y}-\left\vert \lambda \right\vert l_{1}\left( x_{1},x_{2}\right)
\left\Vert x_{1}-x_{2}\right\Vert _{X}.
\end{equation*}

So, under conditions of this theorem conditions (\textit{i})-(\textit{iii})
of the general theorem are fulfilled, and from the above inequality follows
the fulfillment of condition (\textit{iv}) ensure that the image of the $%
F\left( B_{r_{0}}^{X}\left( 0\right) \right) $ is closed. Consequently,
Theorem\ref{Th_5} is proved.
\end{proof}

\begin{remark}
\label{R_5} It needs to note the subset $M\left( \lambda \right) $ defined
in Theorem \ref{Th_5} decreases by increases of the number $\left\vert
\lambda \right\vert \nearrow \lambda _{0}$. The above-mentioned articles
actually seek the number $\lambda _{0}$ that the existence was assumed in
the previous theorem.
\end{remark}

Here we will study how one can determine such numbers that one can use as
the number $\lambda _{0}$, moreover, which are independent of the elements
from the domain of the examined operators. Really by such way, the founded
number can be called the first eigenvalue of the examined operator relative
to the different operator as in Definition \ref{D-S}.

In what follow we will use some results from the article Berger \cite{6}
(see, also \cite{34}). Therefore, we provide here results that are necessary
for us from article \cite{6}.

\begin{definition}
\label{D_B}(\cite{6})Let $A:X\longrightarrow X^{\ast }$ be a variational
operator. Then $A$ is of class $I$ if: \ 

(i) $A$ is bounded, i.e. $\left\Vert A\left( x\right) \right\Vert \leq \mu
\left( \left\Vert x\right\Vert \right) ;$

(ii) $A$ is continuous from the strong topology of $X$ to the weak topology
of $X^{\ast };$

(III) Oddness, i.e. $A\left( -x\right) =-A\left( x\right) $ ;

iv) Coerciveness, i.e. $\overset{1}{\underset{0}{\int }}\ \left\langle
A\left( sx\right) ,x\right\rangle ds\nearrow \infty $ when $\left\Vert
x\right\Vert _{X}\nearrow \infty $;

(V) Monotonicity: $\left\langle A\left( x_{1}\right) -A\left( x_{2}\right)
,x_{1}-x_{2}\right\rangle >0$, for any $x_{1},x_{2}\in X$.
\end{definition}

\begin{lemma}
\label{L_B1}(\cite{6})Let $A$ be as variational operator of class $I$, then
a $\partial A_{R}$ 
\begin{equation*}
\partial A_{R}=\left\{ x\in X\left\vert \ \overset{1}{\underset{0}{\int }}\
\left\langle A\left( sx\right) ,x\right\rangle ds=R\text{\/}\right. \right\}
\end{equation*}%
is a closed, bounded set in $X$. Furthermore $\left\Vert x\right\Vert
_{X}\geq k(R)>0$ and $\partial A_{R}$ is a weakly closed, bounded convex
set, where $k(R)$ is a constant independent of $x\in \partial A_{R}$.
\end{lemma}

\begin{theorem}
\label{Th_6}(\cite{6}Theorem III.I) Let $A:X\longrightarrow X^{\ast }$ be an
operator of class $I$ or $II$ where $X$ is a reflexive Banach space over the
reals. Let $B:X\longrightarrow X^{\ast }$ an operator of class III. Then the
eigenvalue problem $A\left( x\right) =\lambda B\left( x\right) $ has at
least one non-trivial solution. This solution is normalized by the
requirement that $x\in \partial A_{R}$ and characterized as a solution of
the variational problem for%
\begin{equation*}
\sup \left\{ \underset{0}{\overset{1}{\int }}\left\langle B\left( sx\right)
,x\right\rangle ds\left\vert \ x\in \partial A_{R}\right. \right\}
\end{equation*}%
sufficiently large $R$. Furthermore 
\begin{equation*}
\lambda =\frac{\left\langle A\left( x\right) ,x\right\rangle }{\left\langle
B\left( x\right) ,x\right\rangle }
\end{equation*}
\end{theorem}

So, consider the homogeneous equation (\ref{4.1a}) in order to investigate
the existence of the number $\lambda _{0}$.

\begin{proposition}
\label{Pr_2}Let $X\subset Y$ and dense in $Y$, $Z=Y^{\ast }$. Let conditions
1, 2, 4 of the above Theorem \ref{Th_5} are fulfilled for this case and $f$, 
$g$ as the functions are monotone odd functions then there exist such $%
\lambda _{0}>0$ and $x_{\lambda _{0}}\in \partial E^{B_{R_{0}}^{X}\left(
0\right) }\subset X$ that $F_{\lambda _{0}}\left( x_{\lambda _{0}}\right)
\equiv f\left( Lx_{\lambda _{0}}\right) -\lambda _{0}g\left( x_{\lambda
_{0}}\right) =0$, for some number $R_{0}\gg 1$, where $\partial
E^{B_{R_{0}}^{X}\left( 0\right) }$ be defined as follows 
\begin{equation*}
\partial E^{B_{R_{0}}^{X}\left( 0\right) }=\left\{ x\in B_{R_{0}}^{X}\left(
0\right) \subset X\left\vert \ \underset{0}{\overset{1}{\int }}\left\langle
f\left( sLx\right) ,Lx\right\rangle ds=R_{0}\right. \right\} ,
\end{equation*}%
\begin{equation*}
E^{B_{R_{0}}^{X}\left( 0\right) }=\left\{ x\in B_{R_{0}}^{X}\left( 0\right)
\subset X\left\vert \ \underset{0}{\overset{1}{\int }}\left\langle f\left(
sLx\right) ,Lx\right\rangle ds\leq R_{0}\right. \right\}
\end{equation*}%
Moreover, the condition similar to condition 3 of the above theorem
satisfies.
\end{proposition}

\begin{proof}
It is clear that $L\left( B_{R_{0}}^{X}\left( 0\right) \right) $ is convex
due to the linearity of the operator $L$. From above Lemma follows that $%
E^{B_{R_{0}}^{X}\left( 0\right) }$ is a weakly closed, bounded convex set
and $\left\Vert x\right\Vert _{X}\geq k(r_{0})>0$, where $k(r_{0})$ is a
constant independent of $x\in E^{B_{R_{0}}^{X}\left( 0\right) }$, as the
operator $f\circ L$ satisfies all conditions of this lemma.

Consequently, the result of Berger (\cite{6}) follows that $%
E^{B_{R_{0}}^{X}\left( 0\right) }$ and \ $\partial E^{B_{R_{0}}^{X}\left(
0\right) }$ are weakly closed, bounded convex sets.

Now consider the expression $\left\langle g\left( x\right) ,Lx\right\rangle $
and note that there exists such constant $M$ that 
\begin{equation*}
0<\sup \left\{ \left\Vert g\left( x\right) \right\Vert _{Y^{\ast
}}\left\vert \ x\in \partial E^{B_{R_{0}}^{X}\left( 0\right) }\right.
\right\} =M<\infty ,
\end{equation*}%
according to the conditions 1, 2 and boundedness of the norm $\left\Vert
Lx\right\Vert _{Y}$ (as $0<$ $\left\Vert Lx\right\Vert _{Y^{\ast
}}<M_{1}<\infty $).

Consequently, there is constant $\lambda _{0}=\lambda _{0}\left( M\right) $
such that $0<\lambda _{0}<\infty $.
\end{proof}

So, we provide the result on the spectrum of the operator $f\circ L$ respect
to the operator $g$.

\begin{theorem}
\label{Th_8}Let the functions $f$ and $g$ are homogeneous moreover their
order of the homogeneity equally and is the function $\varphi $, i.e. for
each $\tau \in R$ hold the equality $f\left( \tau \cdot y\right) =\varphi
\left( \tau \right) \cdot f\left( y\right) $, $g\left( \tau \cdot y\right)
=\varphi \left( \tau \right) \cdot g\left( y\right) $. Let all conditions of
the above proposition fulfills. Then the operator $f\circ L$ has the
spectrum respect to the operator $g$. Moreover, this spectrum is the
function of the spectrum of the operator $L$.
\end{theorem}

\begin{proof}
From the previous theorem follows the existence of such $\lambda _{0}\in 
\mathbb{%
\mathbb{R}
}$ that the equation $F_{\lambda }\left( \widetilde{x}_{\lambda }\right)
\equiv f\left( L\widetilde{x}_{\lambda }\right) -\lambda g\left( \widetilde{x%
}_{\lambda }\right) =0$ solvable. Then using of the well-known approach it
is necessary to seek an element $x\in X$ that satisfy the following equality 
\begin{equation}
\lambda =\inf \left\{ \frac{\left\langle f\left( Lx\right) ,Lx\right\rangle 
}{\left\langle g\left( x\right) ,Lx\right\rangle }\left\vert \ x\in X\right.
\right\} .  \label{4.5a}
\end{equation}

Due to the cnditions of this theorem, it is enough to study the above
question only for $x\in S_{1}^{X}\left( 0\right) $. We can take into account
that $f$ is an $N-$function and the expression $\left\langle f\left(
Lx\right) ,Lx\right\rangle $ generates a functional $\Phi \left( Lx\right) $
according to the condition on $f$ \footnote{\begin{remark}
In the case when $L$ is the differential operator $\left\langle f\left(
Lx\right) ,Lx\right\rangle $ is a function of the norm $\left\Vert
Lx\right\Vert _{L_{\Phi }}$ of some Lebesgue or Orlicz space, where $\Phi $
be an $N-$function.
\end{remark}
} .

Whence follows that it is enough to seek the number $\lambda $ the following
way 
\begin{equation*}
\lambda =\inf \left\{ \frac{\left\Vert f\left( Lx\right) \right\Vert _{Z}}{%
\left\Vert g\left( x\right) \right\Vert _{Z}}\left\vert \ x\in
S_{1}^{X}\left( 0\right) \right. \right\} .
\end{equation*}

From above theorem follows the existence of the number $\lambda >0$.

Moreover, it isn't difficult to see that all conditions of the Theorem \ref%
{Th_1} are fulfilled under the conditions of this theorem since due to
condition 3 always one can find such a number $\lambda _{0}>0$ that
condition (\textit{iii}) will fulfill. Thus follows the existence an element 
$x_{0}\in S_{1}^{X}\left( 0\right) $ and a number $\lambda _{0}$ that is the
desired infimum.

Let $x_{1}\in S_{1}^{X}\left( 0\right) $ is the first eigenfunction and $%
\lambda _{1}$ first eigenvalue of the operator $L$ then we have 
\begin{equation*}
\lambda _{0}\leq \frac{\varphi \left( \lambda _{1}\right) \left\Vert f\left(
x_{1}\right) \right\Vert _{Z}}{\left\Vert g\left( x_{0}\right) \right\Vert
_{Z}}.
\end{equation*}
\end{proof}

Now we provide some examples of operators related to the above theorems.

1. Let $L:W^{m,p}\left( \Omega \right) \longrightarrow L_{p}\left( \Omega
\right) $ be a linear differential operator with the spectrum $P\left(
L\right) \subset R_{+}$, the operator $f$ is the function $f\left( \tau
\right) =\left\vert \tau \right\vert ^{p-2}\tau $ and $g\equiv f$. So, it
needs to define the first eigenfunction and eigenvalue of the operator $%
f\left( L\circ \right) $ relative to operator $g\left( \circ \right) $. Then
using the expression (\ref{4.5a}) we get 
\begin{equation*}
\lambda _{f}=\inf \left\{ \frac{\left\langle f\left( Lu\right)
,Lu\right\rangle }{\left\langle g\left( u\right) ,Lu\right\rangle }%
\left\vert \ u\in W^{m,p}\left( \Omega \right) \right. \right\} =
\end{equation*}%
\begin{equation*}
=\inf \left\{ \frac{\left\Vert Lu\right\Vert _{L_{p}}^{p}}{\underset{\Omega }%
{\dint }\left\vert u\right\vert ^{p-2}u\cdot Ludx}\left\vert \ u\in
W^{m,p}\left( \Omega \right) \right. \right\} =
\end{equation*}%
\begin{equation*}
=\inf \left\{ \frac{\left\Vert Lu\right\Vert _{L_{p}}^{p-1}}{\left\Vert
u\right\Vert _{L_{p}}^{p-1}}\left\vert \ u\in W^{m,p}\left( \Omega \right)
\right. \right\} =
\end{equation*}%
\begin{equation*}
=\inf \left\{ \left( \frac{\left\Vert Lu\right\Vert _{L_{p}}}{\left\Vert
u\right\Vert _{L_{p}}}\right) ^{p-1}\left\vert \ u\in S_{1}^{W^{m,p}\left(
\Omega \right) }\left( 0\right) \right. \right\} .
\end{equation*}%
Whence we arrive that $\lambda _{f1}=\lambda _{L1}^{p-1}$, where $\lambda
_{L1}$ is the first eigenvalue and the function $u_{1}\in
S_{1}^{W^{m,p}\left( \Omega \right) }\left( 0\right) $ is the first
eigenfunction of the operator $L$.

2. We will study the spectral property of the fully nonlinear operator in
the following two special cases 
\begin{equation}
-\left\vert \Delta u\right\vert ^{p-2}\Delta u=\lambda \left\vert \nabla
u\right\vert ^{\mu -2}u,\quad x\in \Omega ,\quad u\left\vert ~_{\partial
\Omega }\right. =0,  \tag{4.2}
\end{equation}%
\begin{equation}
-\left\vert \Delta u\right\vert ^{p-2}\Delta u=\lambda \left\vert
u\right\vert ^{\nu }u,\quad x\in \Omega ,\quad u\left\vert ~_{\partial
\Omega }\right. =0,  \tag{4.3}
\end{equation}%
i.e. we will seek of the spectrum of the operator $-\left\vert \Delta
u\right\vert ^{p-2}\Delta u$ relatively of operators $\left\vert \nabla
u\right\vert ^{\mu -2}u$ and $\left\vert u\right\vert ^{\nu }u$, separately.

2 \textbf{(}\textit{a}\textbf{).} Consider problem (4.2). For study of the
posed question we will use the following equality 
\begin{equation*}
\left\langle -\left\vert \Delta u\right\vert ^{p-2}\Delta u,-\Delta
u\right\rangle =\left\langle \lambda \left\vert \nabla u\right\vert ^{\mu
-2}u,-\Delta u\right\rangle
\end{equation*}%
then we get%
\begin{equation*}
\left\Vert \Delta u\right\Vert _{p}^{p}=\lambda \left( \mu -1\right)
^{-1}\left\Vert \nabla u\right\Vert _{\mu }^{\mu }.
\end{equation*}

Hence follows 
\begin{equation*}
\lambda _{1}\left( p,\mu \right) =\left( \mu -1\right) \inf \ \left\{ \frac{%
\left\Vert \Delta u\right\Vert _{p}^{p}}{\left\Vert \nabla u\right\Vert
_{\mu }^{\mu }}\left\vert \ u\in W^{2,p}\cap W_{0}^{1,p}\right. \right\} =\ 
\end{equation*}%
\begin{equation*}
\left( \mu -1\right) \inf \left\{ \frac{\left\Vert \Delta u\right\Vert _{p}}{%
\left\Vert \nabla u\right\Vert _{\mu }^{\mu /p}}\left\vert \ u\in
W^{2,p}\cap W_{0}^{1,p}\right. \right\}
\end{equation*}%
again for to find $\lambda $ satisfying of the assumed condition we must
select such exponent $\mu $ that the given case requires.

Consequently, we need assume $\mu =p$, then we get 
\begin{equation}
\lambda _{1}\left( p,p\right) =\left( p-1\right) \inf \ \left\{ \frac{%
\left\Vert \Delta u\right\Vert _{p}}{\left\Vert \nabla u\right\Vert _{p}}%
\left\vert \ u\in W^{2,p}\left( \Omega \right) \cap W_{0}^{1,p}\left( \Omega
\right) \right. \right\} .  \label{4.4}
\end{equation}

As is well-known $\left\Vert \nabla u\right\Vert _{p}\leq c\left( p,\Omega
\right) \left\Vert \Delta u\right\Vert _{p}$ under the condition $%
u\left\vert \ _{\partial \Omega }\right. =0$, consequently, $\lambda
_{1}\left( p,p\right) \leq \left( p-1\right) c\left( p,\Omega \right) $.

\begin{proposition}
\label{Pr_5}Let $f_{0}:W^{2,p}\left( \Omega \right) \cap W_{0}^{1,p}\left(
\Omega \right) \longrightarrow L^{q}\left( \Omega \right) $ that has the
presentation $f_{0}\left( u\right) =-\left\vert \Delta u\right\vert
^{p-2}\Delta u$ and $f_{1}:W_{0}^{1,p}\left( \Omega \right) \longrightarrow
L^{q}\left( \Omega \right) $ that has the presentation $f_{1}\left( u\right)
=\left\vert \nabla u\right\vert ^{p-2}u$ . Then $f_{0}$ has minimal spectrum
with respect to $f_{1}$ which defined by (4.4).
\end{proposition}

2 (\textit{b}). Consider problem (4.3) for $\nu =p-2$ then we have%
\begin{equation*}
\left\langle -\left\vert \Delta u\right\vert ^{p-2}\Delta u,-\Delta
u\right\rangle =\left\langle \lambda \left\vert u\right\vert ^{p-2}u,-\Delta
u\right\rangle \Longrightarrow
\end{equation*}%
\begin{equation*}
\left\Vert \Delta u\right\Vert _{p}^{p}=\lambda \frac{4\left( p-1\right) }{%
p^{2}}\left\Vert \nabla \left( \left\vert u\right\vert ^{\frac{p-2}{2}%
}u\right) \right\Vert _{2}^{2}
\end{equation*}%
or 
\begin{equation*}
\left\Vert \Delta u\right\Vert _{p}^{p}=\lambda \left( p-1\right) \left\Vert
\left( \left\vert u\right\vert ^{p-2}\left\vert \nabla u\right\vert
^{2}\right) \right\Vert _{1}.
\end{equation*}%
Thus we get 
\begin{equation*}
\widetilde{\lambda }_{1}\left( p\right) =\frac{1}{p-1}\inf \left\{ \ \frac{%
\left\Vert \Delta u\right\Vert _{p}^{p}}{\left\Vert \left\vert u\right\vert
^{\frac{p-2}{2}}\left\vert \nabla u\right\vert \right\Vert _{2}^{2}}%
\left\vert \ u\in W^{2,p}\cap W_{0}^{1,p}\right. \right\} =
\end{equation*}%
\begin{equation}
=\frac{p}{2\left( p-1\right) }\inf \left\{ \ \frac{\left\Vert \left\vert
\Delta u\right\vert ^{\frac{p}{2}}\right\Vert _{2}}{\left\Vert \nabla \left(
\left\vert u\right\vert ^{\frac{p-2}{2}}\left\vert u\right\vert \right)
\right\Vert _{2}}\left\vert \ u\in W^{2,p}\cap W_{0}^{1,p}\right. \right\}
\label{4.5}
\end{equation}%
Hence we obtain 
\begin{equation}
\widetilde{\lambda }_{1}\left( p\right) \geq \frac{1}{p-1}\frac{\left\Vert
\Delta u\right\Vert _{p}^{p}}{\left\Vert u\right\Vert _{p}^{p-2}\left\Vert
\nabla u\right\Vert _{p}^{2}}  \label{4.6}
\end{equation}%
according of the following inequality 
\begin{equation*}
\left\Vert \left( \left\vert u\right\vert ^{p-2}\left\vert \nabla
u\right\vert ^{2}\right) \right\Vert _{1}\leq \left\Vert u\right\Vert
_{p}^{p-2}\left\Vert \nabla u\right\Vert _{p}^{2}.
\end{equation*}

So, we arrive to result

\begin{proposition}
\label{Pr_3}Let $f_{0}:W^{2,p}\left( \Omega \right) \cap W_{0}^{1,p}\left(
\Omega \right) \longrightarrow L^{q}\left( \Omega \right) $ that has the
presentation $f_{0}\left( u\right) =-\left\vert \Delta u\right\vert
^{p-2}\Delta u$ and $f_{1}:L^{p}\left( \Omega \right) \longrightarrow
L^{q}\left( \Omega \right) $ that has the presentation $f_{1}\left( u\right)
=\left\vert u\right\vert ^{p-2}u$ . Then $f_{0}$ has minimal spectrum with
respect to $f_{1}$ which defined by (4.5) and satisfies the inequation (4.6).
\end{proposition}

\label{R_3}One can approach at the mentioned above question on the relation
using also of the following equality 
\begin{equation*}
-\int \left\vert \Delta u\right\vert ^{p-2}\Delta u=\lambda \int \left\vert
u\right\vert ^{p-2}u\Longrightarrow
\end{equation*}%
as the operator $-\Delta $ is positive under the conditions of this section,
then we get 
\begin{equation*}
\left\Vert \Delta u\right\Vert _{p-1}=\lambda \left\Vert u\right\Vert
_{p-1}\Longrightarrow
\end{equation*}%
\begin{equation*}
\lambda =\inf \left\{ \frac{\left\Vert \Delta u\right\Vert _{p-1}}{%
\left\Vert u\right\Vert _{p-1}}\left\vert \ u\in W^{2,p}\cap
W_{0}^{1,p}\right. \right\} .
\end{equation*}

\subsubsection{Some remarks on the eigenvalues and bifurcation}

From above-mentioned results follows:

1. To seek the eigenvalues of the nonlinear continuous operator in the
Banach space is necessary to choose the other operator in such a way that
the order of nonlinearity will be an identical whit the order of
nonlinearity of the examined operator. If one uses the proposed approach
then is possible to find and the other eigenvalues of this operator.

2. To study the bifurcation of solutions, in the beginning, is necessary to
find eigenvalues of this operator relatively to the nonlinear operator that
has identical order of the nonlinearity with the operator from the main
part, moreover, the choosing another operator must take account of the
properties of the second part of the examined equation, e.g. if consider the
equation (\ref{2.1a}) usually assumed $f\succ g$, where $f$ is the main part
and $g$ is the second part of this equation, but in this case, the order of
nonlinearity $g$ must be greater than the order of nonlinearity $f$ (see,
e.g. Section 3), as in articles \cite{12, 16, 33}, etc.

\end{document}